%% file: Main.tex
\documentclass[reqno, oneside]{amsart}
\usepackage{etex}
\usepackage{chemarr}
\usepackage{hyperref}
\usepackage{comment}

\usepackage[a4paper]{geometry}

\usepackage{latexsym, amscd, amsfonts, eucal, mathrsfs, amsmath, 
amssymb, mathabx, amsthm, xr, stmaryrd, color, enumerate, tikz}
\usepackage{appendix}
\usepackage{multicol}
\usepackage[all]{xy}
\usepackage{hyperref}
\usepackage{fullpage}

\newtheorem{theorem}{Theorem}[section]
\newtheorem{lemma}[theorem]{Lemma}
\newtheorem{proposition}[theorem]{Proposition}
\newtheorem{corollary}[theorem]{Corollary}

\theoremstyle{definition}
\newtheorem{definition}[theorem]{Definition}
\newtheorem{construction}[theorem]{Construction}

\newtheorem{example}[theorem]{Example}

\newtheorem{warning}[theorem]{Warning}
\newtheorem{remark}[theorem]{Remark}
\newtheorem{question}[theorem]{Question}

\theoremstyle{definition}
\newtheorem{nul}{}[section]
\newtheorem{dfn}[nul]{Definition}

\newtheorem{rmk}[nul]{Remark}

\newtheorem*{obs}{Observation}

\newtheorem*{dfn*}{Definition}
\newtheorem*{axm*}{Axiom}
\newtheorem*{ntn*}{Notation}
\newtheorem*{exm*}{Example}
\newtheorem*{exr*}{Exercise}
\newtheorem*{int*}{Intuition}
\newtheorem*{qst*}{Question}
\newtheorem*{rmk*}{Remark}

\theoremstyle{plain}

\newtheorem{thm}{Theorem}

\newtheorem*{thm*}{Theorem}
\newtheorem*{prop*}{Proposition}
\newtheorem*{cor*}{Corollary}
\newtheorem*{lem*}{Lemma}
\newtheorem*{cnj*}{Conjecture}

\DeclareMathOperator{\smsh}{\wedge}

\usepackage{tikz}
\usetikzlibrary{matrix,arrows,decorations}
\usepackage{tikz-cd}

\usepackage{adjustbox}

\begin{document}

\title{Eilenberg-Maclane spectra as equivariant Thom spectra}
\author{Jeremy Hahn and Dylan Wilson}

\begin{abstract} We prove that the $G$-equivariant mod $p$ Eilenberg--MacLane spectrum arises as an equivariant Thom spectrum for any finite, $p$-power cyclic group $G$, generalizing a result of Behrens and the second author in the case of the group $C_2$.  We also establish a construction of $\mathrm{H}\underline{\mathbb{Z}}_{(p)}$, and prove intermediate results that may be of independent interest. Highlights include constraints on the Hurewicz images of equivariant spectra that admit norms, and an analysis of the extent to which the non-equivariant $\mathrm{H}\mathbb{F}_p$ arises as the Thom spectrum of a more than double loop map.
\end{abstract}


\setcounter{tocdepth}{1}
\maketitle

\tableofcontents

\vbadness 5000


\input{Introductionv3.tex}

\input{Outline.tex}

\input{ActionsOnQuaternionicProjectiveSpace.tex}

\input{Prime2.tex}

\input{HSpaceOrientation.tex}

\input{NormArgument.tex}

\input{NormArgumentII.tex}

\input{TheProof.tex}

\input{Epilogue.tex}

\appendix

\input{appendix.tex}

\bibliographystyle{amsalpha}
\nocite{*}
\bibliography{Bibliography}

\end{document}

%% file: Introductionv3.tex
\section{Introduction}

Both authors are fond of the following result of Mahowald \cite{mahowald}:

\begin{theorem}[Mahowald] \label{thm:MahE2}
The Thom spectrum of the unique non-trivial double loop map
\begin{equation} \label{eqn:Mah}
\Omega^2 S^3 \longrightarrow \mathrm{BO}
\end{equation}
is $\mathrm{H}\mathbb{F}_2$.
\end{theorem}

This result and several variants have enjoyed many subsequent proofs in the literature.  Examples include \cite{mahowald-thom,cohen-may-taylor,priddy,temss,omar-toby,akhil-niko-justin}.  In joint work with Mark Behrens \cite{behrens-wilson}, the second author proved a $C_2$-equivariant generalization of Mahowald's Theorem \ref{thm:MahE2}:

\begin{theorem}[Behrens--Wilson] \label{thm:BehrensWilson}
Let $\rho$ denote the real regular representation of the group $C_2$.  Then there is a $\Omega^{\rho}$-map
$$\Omega^{\rho} S^{\rho+1} \to \mathrm{BO}_{C_2}$$
with Thom spectrum $\mathrm{H}\underline{\mathbb{F}}_2$, the Eilenberg--Maclane object associated to the constant Mackey functor $\underline{\mathbb{F}}_2$.
\end{theorem}

A less well-known fact is that Mahowald's map (\ref{eqn:Mah}) is in fact a triple loop map, as the authors first learned from Mike Hopkins:

\begin{obs} \label{obs:TripleLoop}
Mahowald's map (\ref{eqn:Mah}) may be obtained by thrice looping the composite
$$\mathbb{H}P^{\infty} \simeq \mathrm{BSp}(1) \longrightarrow \mathrm{BSp}
\simeq \mathrm{B}^5\mathrm{O} \stackrel{\eta}{\longrightarrow} \mathrm{B}^4\mathrm{O}.$$
\end{obs}

The present work arose from the authors' attempt to understand how the above
observation
generalizes to the $C_2$-equivariant setting.  We obtain the following theorem:

\begin{thm} \label{thm:main-C2}
Let $\sigma$ denote the sign representation of the group $C_2$.  There is a $C_2$-action on the space $\mathbb{H}P^{\infty}$ and a $\Omega^{\rho+\sigma}$-map
$$\Omega^{\rho+\sigma} \mathbb{H}P^{\infty} \to \mathrm{BO}_{C_2},$$
with Thom spectrum $\mathrm{H}\underline{\mathbb{F}}_2$.
\end{thm}

Since $C_2$ is the only group with a $2$-dimensional real regular representation, the statement of Theorem \ref{thm:BehrensWilson} makes it unclear how the constant $G$-Mackey functor $\mathrm{H}\underline{\mathbb{F}}_2$ might be an equivariant Thom spectrum for any larger group $G$.  On the other hand, the $C_2$-representation $\rho+\sigma = 2\sigma+1$ is naturally the restriction of a $C_4$-representation.  In this paper, we in fact obtain Theorem \ref{thm:main-C2} as a special case of the following more general result:

\begin{thm} \label{thm:main-even}
Fix an integer $n \ge 0$, and let $G=C_{2^n}$ denote the cyclic group of order $2^n$.  Let $\lambda$ denote the \textit{standard} representation of $G$ on the complex plane, where the generator acts by $e^{2 \pi i /2^n}$.  Then there is a $G$-action on $\mathbb{H}P^{\infty}$, and a $\Omega^{\lambda+1}$-map
$$\Omega^{\lambda+1} \mathbb{H}P^{\infty} \to \mathrm{BO}_{G},$$
with Thom spectrum $\mathrm{H}\underline{\mathbb{F}}_2$.
\end{thm}

\begin{remark}
The group $G=C_2$ is special: only for this group does $\lambda=2\sigma$ split as the sum of two smaller representations.  This leads to a non-obvious equivalence of $C_2$-spaces
$$\Omega^{\rho} S^{\rho+1} = \Omega^{\sigma+1} S^{\sigma+2} \simeq \Omega^{2\sigma+1} \mathbb{H}P^{\infty} \simeq \Omega^{2 \sigma} S^{2 \sigma+1} = \Omega^{\lambda} S^{\lambda+1},$$
which implies a $\Omega^{\lambda}$-analog of Theorem \ref{thm:BehrensWilson} (for details, see Section \ref{sec:action}).
A Borel equivariant version of this $\Omega^{\lambda}$-analog previously appeared as Lemma $3.1$ in \cite{klang}, where it is used in an essential manner to compute the factorization homology of Eilenberg--Maclane spectra.
\end{remark}

We next turn to odd primes $p>2$.  So long as one is willing to contemplate Thom spectra of $p$-local spherical fibrations (as in, e.g., \cite[\S 3.4]{bcs}), Mahowald's Theorem \ref{thm:MahE2} admits an analog due to Mike Hopkins \cite[Theorem 4.18]{akhil-niko-justin}:

\begin{theorem}[Hopkins] \label{thm:Hop}
Let 
	\[
	S^3 \stackrel{1-p}{\longrightarrow} \mathrm{B}^3\mathrm{GL}_1(S^0_{(p)})
	\]
denote the class $1-p \in \pi_0(S_{(p)})^{\times} = 
\pi_3(\mathrm{B}^3\mathrm{GL}_1(S_{(p)})).$  Then, applying $\Omega^2$, one obtains a map
	\[
	\Omega^2 S^3 \longrightarrow \mathrm{BGL}_1(S^0_{(p)})
	\]
with Thom spectrum equivalent to $\mathrm{H}\mathbb{F}_p$.
\end{theorem}

In light of the situation at $p=2$, it is natural to wonder if the map
$$1-p:S^3 \longrightarrow \mathrm{B}^3 \mathrm{GL}_1(S^0_{(p)})$$
may be delooped.  The authors were surprised to find that it \textit{cannot}, which we record as our only non-equivariant result:

\begin{thm} \label{thm:non-eqvt}
Let $S^0_{(p)}$ denote the $p$-local sphere spectrum, and suppose $p>2$.  Then there is no triple loop map
$$X \to \mathrm{BGL}_1(S^0_{(p)}),$$
for any triple loop space $X$, with Thom spectrum $\mathrm{H}\mathbb{F}_p$.  The same is true if $S^0_{(p)}$ is replaced by the $p$-completed sphere spectrum.
\end{thm}

\begin{remark}
At the prime $p=3$, the map 
$$1-p:S^3 \longrightarrow \mathrm{B}^3 \mathrm{GL}_1(S^0_{(p)})$$
is not even an $H$-space map for the standard $H$-space structure on $S^3=\mathrm{SU}(2)$.  Crucially, we will see in Section \ref{sec:exotic} that the map \emph{is} an $H$-space map for a certain exotic $H$-space structure on $S^3$.  Determining the maximum amount of structure present on this map remains an interesting question.
\end{remark}

Our equivariant generalization of Theorem \ref{thm:Hop} is as follows:

\begin{thm}[Theorem \ref{thm:main}] \label{thm:MainThom}
Fix an integer $n \ge 0$ and a prime number $p$, and let $G$ denote the cyclic group $C_{p^n}$.  Let $\lambda$ denote the standard representation of $G$ on the complex numbers, where a generator acts by $e^{2\pi i / p^n}$, and let $S^0_{(p)}$ denote the $G$-equivariant sphere spectrum.  Finally, let
	\[
	\mu: \Omega^{\lambda}S^{\lambda+1}
	\longrightarrow \mathrm{BGL}_1(S^0_{(p)})
	\]
denote the $\Omega^{\lambda}$-map obtained by applying $\Omega^{\lambda}$ to the map
	\[
	S^{\lambda+1} \longrightarrow \mathrm{B^{\lambda+1}GL}_1(S^0_{(p)})
	\]
corresponding to $1-p \in (\pi_0^{G}S^0_{(p)})^{\times}$.
Then the Thom spectrum $\left(\Omega^{\lambda}S^{\lambda+1}\right)^\mu$
of $\mu$ is $\mathrm{H}\underline{\mathbb{F}}_p$.
\end{thm}

\begin{remark}
We will include a detailed discussion of Thom spectra of $G$-equivariant, $p$-local spherical fibrations in Section \S \ref{sec:prime2}.  This will include in particular a discussion of $\mathrm{BGL}_1(S^0_{(p)})$ as a $G$-infinite loop space, allowing us to form its $\lambda$-delooping $\mathrm{B^{\lambda+1}GL}_1(S^0_{(p)})$. The basic definitions are due to Lewis and May, and found in \cite[Chapter X]{LMS}.
\end{remark}

We also establish a ($p$-local) integral variant of Theorem \ref{thm:MainThom}:

\begin{thm}[Theorem \ref{thm:Z}] \label{thm:introZ} Let $S^{\lambda+1}\langle \lambda+1\rangle$
denote the fiber of the unit
	\[
	S^{\lambda+1} \longrightarrow
	\Omega^{\infty}\left(\Sigma^{\lambda+1}
	\mathrm{H}\underline{\mathbb{Z}}\right).
	\]
Then there is an equivalence
	\[
	\left(\Omega^{\lambda}
	(S^{\lambda+1}\langle \lambda+1\rangle)\right)^{\mu}
	\simeq \mathrm{H}\underline{\mathbb{Z}}_{(p)}.
	\]
\end{thm}

\begin{remark}
Collections of little disks in the representation $\lambda$ form an equivariant operad $\mathcal{O}_{\lambda}$, leading to the notion of an $\mathbb{E}_{\lambda}$-algebra \cite{hauschild}.  A nice modern discussion of $\mathbb{E}_\lambda$-algebras may be found in \cite{hill-disks}.

If $X$ is any pointed $G$-space, then $\Omega^{\lambda} X$ is naturally an $\mathbb{E}_\lambda$-algebra in $G$-spaces.  If $X$ is a $G$-connected, pointed $G$-space, then the free $\mathbb{E}_\lambda$-algebra on $X$ is $\Omega^{\lambda} \Sigma^{\lambda} X$ \cite[Theorem 1]{rs}. Any $G$-$\mathbb{E}_\infty$-ring, such as the Eilenberg--Maclane object $\mathrm{H}\underline{\mathbb{Z}}_{(p)}$
\cite[\S 5.7]{schwede}, is naturally an $\mathbb{E}_\lambda$-algebra object in $G$-Spectra.  It was proved by Lewis \cite[Remark X.6.4]{LMS} that the Thom spectrum of a $\Omega^{\lambda}$-map is naturally an $\mathbb{E}_{\lambda}$-algebra in $G$-spectra, and it follows from our proof that the equivalence of Theorem \ref{thm:introZ} is one of $\mathbb{E}_\lambda$-algebras.
\end{remark}

\begin{remark}
Just as Mahowald's original Theorem \ref{thm:MahE2} may be phrased as the fact that $\mathrm{H}\mathbb{F}_2$ is the free $\mathbb{E}_2$-algebra with a nullhomotopy $2 \simeq 0$ \cite{akhil-niko-justin, omar-toby}, our Theorem \ref{thm:MainThom} may be read as the statement that the free $\mathbb{E}_\lambda$-algebra with $p \simeq 0$ is $\mathrm{H}\underline{\mathbb{F}}_p$.
\end{remark}

\subsubsection*{Conventions} We freely use the language
of $\infty$-categories (alias quasicategories,
alias weak Kan complexes) in the form developed
in \cite{HTT} (though we give specific references within \cite{HTT} whenever technical results are applied).  We denote by $\mathsf{Spaces}$ the $\infty$-category
of spaces, and append decorations to obtain the
$\infty$-category of pointed spaces or of $G$-spaces
for a finite group $G$,
the latter being defined as the $\infty$-category
of presheaves $\mathsf{Psh}(\mathcal{O}_G)$ on the
1-category of transitive $G$-sets. We denote
by $\mathsf{Sp}^G$ the $\infty$-category of
(genuine) $G$-spectra, and assume the reader
is familiar with the standard notations of equivariant
homotopy theory. We could not improve on the
summary given in
\cite[\S2-3]{HHR}, and recommend it to the reader. In particular,
we will require the Eilenberg-MacLane
spectrum associated to the constant Mackey
functor $\underline{\mathbb{F}}_p$, and the
notion of the geometric fixed points $X^{\Phi G}$
of a $G$-spectrum. Finally, we denote the $G$-space
classifying equivariant stable spherical fibrations
with fiber of type $S^0$ by
$\mathrm{BGL}_1(S^0)$. (See \S\ref{sec:prime2} 
for more details).

\subsubsection*{Acknowledgements}
The authors would like to thank Mark Behrens, Jun-Hou Fung, Mike Hopkins, Inbar Klang, and Ishan Levy for helpful conversations related to this paper. We also thank
Nick Kuhn and Peter May for comments on an earlier draft.  Special thanks are due to the anonymous referee, whose painstaking work has led to enormous improvement of the exposition.

%% file: Outline.tex
\section{Outline of the proof of Theorem \ref{thm:MainThom}} \label{sec:outline}

We fix, for the remainder of the paper, a prime $p$ as well as a non-negative integer $n$.  We let $G=C_{p^n}$ denote the cyclic group of order $p^n$.

The Thom spectrum of the map 
$$S^1 \stackrel{1-p}{\longrightarrow} \mathrm{BGL}_1(S^0_{(p)})$$
is the mod $p$ Moore spectrum $M(p)$, which admits a Thom class
for $\mathrm{H}\underline{\mathbb{F}}_p$. It follows formally
(see the argument for Proposition 5.3 in \cite{behrens-wilson})
that there is a Thom class
	\[
	\alpha: \left(\Omega^{\lambda}S^{\lambda+1}\right)^{\mu}
	\longrightarrow \mathrm{H}\underline{\mathbb{F}}_p.
	\]
Our goal is to show that this map is an equivalence of $G$-spectra.
By induction on the order of the group
(the base case being supplied by Hopkins-Mahowald),
it will suffice to prove that the map on geometric fixed
points
	\[
	\alpha^{\Phi G}: 
	\left(\left(\Omega^{\lambda}S^{\lambda+1}\right)^{\mu}\right)^{\Phi G}
	\longrightarrow \mathrm{H}\underline{\mathbb{F}}^{\Phi G}_p
	\]
is an equivalence.

The proof proceeds in several steps.
	\begin{itemize}
	\item[\underline{Step 1}.] Compute the homotopy groups
	of $\left(\left(\Omega^{\lambda}
	S^{\lambda+1}\right)^{\mu}\right)^{\Phi G}$.
	\item[\underline{Step 2}.] Compute the homotopy groups of
	$\mathrm{H}\underline{\mathbb{F}}^{\Phi G}_p$.
	\item[\underline{Step 3}.] Show that $\alpha^{\Phi G}$ 
	is a ring map (for some ring structure on the source).
	\item[\underline{Step 4}.]
	Show that $\alpha^{\Phi G}$ hits algebra generators
	in the target.
	\end{itemize}
	
The computation in Step 2 is well-known, and is stated as
Lemma \ref{lem:em-geo-fixed} below. At odd primes, one learns that
	\[
	\pi_*\mathrm{H}\underline{\mathbb{F}}_p^{\Phi G}
	= \mathbb{F}_p[t] \otimes \Lambda(s), \, |t|=2, |s|=1.
	\]
Step 1 is more difficult. In \S\ref{sec:action} we show that
$\left(\Omega^{\lambda}S^{\lambda+1}\right)^{G} =
\Omega^2S^3 \times \Omega S^2$, and,
after developing some more properties
of this Thom spectrum,
we verify (Lemma \ref{lem:htpy-thom-geo-fixed}) that the homotopy groups
of
$\left(\left(\Omega^{\lambda}S^{\lambda+1}\right)^{\mu}\right)^{\Phi G}$
are \emph{additively} given by
	\[
	\pi_*\left(\left(\Omega^{\lambda}S^{\lambda+1}\right)^{\mu}\right)^{\Phi G}
	= \mathbb{F}_p[x] \otimes \Lambda(y),\, |x|=2, |y|=1.
	\] 

There is a serious problem to be addressed in Step 3: by construction,
$\left(\Omega^{\lambda}S^{\lambda+1}\right)^{\mu}$ is an
$\mathbb{E}_{\lambda}$-spectrum and the Thom class is represented
by an $\mathbb{E}_{\lambda}$-map. When we take geometric fixed points,
we are left without any obvious multiplicative structure on either the source or the map.  We must do additional work to equip with $\alpha^{\Phi G}$ with additional multiplicative structure.  The weakest structure that suffices for our purposes is that of an $\mathbb{A}_2$-algebra; we recall the definition below:

\begin{dfn}\label{dfn:a2}
An $\mathbb{A}_2$-algebra in a symmetric monoidal $\infty$-category $(\mathcal{C},\otimes,\mathbf{1})$ consists of the following data:
\begin{enumerate}
\item An object $X \in \mathcal{C}$.
\item A \emph{multiplication} map $m:X \otimes X \to X$.
\item A \emph{unit} map $u:\mathbf{1} \to X$.
\item Choices of $2$-simplices filling each of the diagrams
$$
\begin{tikzcd}
X \otimes 1 \arrow[swap]{d}{id. \otimes u} \arrow{r}{\simeq} &X &&& 1 \otimes X \arrow[swap]{d}{u \otimes id.} \arrow{r}{\simeq} &X\\ 
X \otimes X, \arrow[swap]{ur}{m} &&&& X \otimes X. \arrow[swap]{ur}{m}
\end{tikzcd}
$$
\end{enumerate}
\end{dfn}

\begin{rmk}
An $\mathbb{A}_2$-algebra in pointed spaces (with the cartesian symmetric monoidal structure) is an $H$-space.
\end{rmk}

To accomplish Step 3, we show that $\alpha^{\Phi G}$ is an $\mathbb{A}_2$-algebra morphism for a certain $\mathbb{A}_2$-structure on the source.
To do so requires a different argument at the prime 2
than at odd primes.
	\begin{itemize}
	\item[\underline{Step 3a})] In \S\ref{sec:action} we show that
	$S^{\lambda+1}= \Omega \mathbb{H}P^{\infty}$ for
	a certain $G$-action on $\mathbb{H}P^{\infty}$. 
	We then show (\S\ref{sec:prime2}) that, when $p=2$, the map
		\[
		S^{\lambda+1} \longrightarrow \mathrm{B}^{\lambda}
		\mathrm{BGL}_1(S^{0}_{(2)})
		\]
	deloops once. Thus, $\alpha^{\Phi G}$ is an $\mathbb{A}_{\infty}$-map
	in this case. This section contains some material that may be of
	independent interest, such as a description of one of the spaces
	in the equivariant $K$-theory spectrum in terms of bundles of (twisted)
	$G$-$\mathbb{H}$-modules. 
	\item[\underline{Step 3b})] 
	In \S\ref{sec:exotic} we produce an exotic $\mathbb{A}_2$-structure
	on $S^{\lambda+1}$ at odd primes with respect to which
	the map $\mu$ is an $\mathbb{A}_2$-map in 
	$\mathbb{E}_{\lambda}$-spaces.
	Our proof uses a small dose of unstable equivariant homotopy theory,
	in particular the EHP sequence.
	\end{itemize}
	
Finally, we come to Step 4. The element 
$s \in \pi_1\mathrm{H}\underline{\mathbb{F}}_p^{\Phi G}$
arises as a witness to the fact that the composite
	\[
	S^0 \stackrel{\nabla}{\longrightarrow}
	G_+ \stackrel{1}{\longrightarrow} \mathrm{H}\underline{\mathbb{F}}_p
	\]
is null. Said differently, $s$ witnesses the vanishing of the element
$[G] \in \pi_0^{G}S^0$ in the Hurewicz image. Nonequivariantly,
the zeroth homotopy group of $\left(\Omega^2S^3\right)^{\mu}$
is already detected by the map
	\[
	(S^1)^{\mu} = M(p) \longrightarrow \left(\Omega^2S^3\right)^{\mu}.
	\]
Equivariantly, this is no longer true. Recall \cite{segal, tomdieck}
that $\pi_0^{G}S^0 = A(G)$ is the Burnside ring of finite $G$-sets.
While the element $p$ dies in the Moore space $(S^1)^{\mu}$,
the same is not true of the elements $[G/K]$ for $K \subsetneq G$. 
The proof that these elements vanish in the Hurewicz image of
$\left(\Omega^{\lambda}S^{\lambda+1}\right)^{\mu}$ is
a consequence of the vanishing of $p$
together with the existence of \emph{norms} supplied by
the $\mathbb{E}_{\lambda}$-structure. This is proved in
\S\ref{sec:norm} and the final pieces of the proof of Step 4
and the main result are then spelled out in \S\ref{sec:proof}.

We end in \S\ref{sec:epilogue} with an explanation of how to produce 
	$\mathrm{H}\underline{\mathbb{Z}}_{(p)}$ as a Thom spectrum
	as well as some unanswered questions.

%% file: ActionsOnQuaternionicProjectiveSpace.tex
\section{Quaternionic projective space}\label{sec:action}

We begin by examining a very natural action of $S^1$ on
quaternionic projective space. It is in some ways analogous
to the natural action of $C_2$ on $\mathbb{C}P^{\infty}$
by complex conjugation.

\begin{construction}\label{cstr:quat-action} Consider the action of
$\mathrm{Sp}(1)$ on the quaternions $\mathbb{H}$
by conjugation. The center $\{\pm 1\}$ acts trivially,
so this produces an action of $\mathrm{Sp}(1)/\{\pm1\}
= \mathrm{SO}(3)$ on $\mathbb{H}$. This produces
an action of $S^1 \subseteq SO(3)$ on $\mathbb{H}$
which can be described in two equivalent ways:
	\begin{itemize}
	\item $z \in S^1$ acts on $q \in \mathbb{H}$
	by $(\sqrt{z}) q (\sqrt{z})^{-1}$;
	\item $z \in S^1$ acts on $q = u+vj$
	by $u+ (zv)j$, where $u,v \in \mathbb{C}$.
	\end{itemize}
From the first description it's clear that $\mathbb{H}$
is an $S^1$-equivariant algebra. From the second
description it's clear that $\mathbb{H} = 2+\lambda$
as a representation.
\end{construction}

\begin{definition}
Let $\mathbb{H}P^{\infty}$ denote the $S^1$-space
obtained from $\mathbb{H}^{\infty}$ by imposing the
relation
	\[
	[q_0: q_1: \cdots] \sim [q_0h: q_1h: \cdots], \, h \in \mathbb{H}^{\times}
	\]
and acting by $S^1$ componentwise as in Construction \ref{cstr:quat-action}.
\end{definition}
We will, without comment, also denote by $\mathbb{H}P^{\infty}$
the $G$-space obtained by restricting the action to a finite subgroup
$G \subseteq S^1$. We will assume that $G$ is a fixed,
\emph{nontrivial} finite subgroup of $S^1$ for the remainder of
this section.

\begin{remark}\label{rmk:quat-fixed} It follows from the definition
that the natural inclusion $\mathbb{C}P^{\infty} \subseteq \mathbb{H}P^{\infty}$
is equivariant for the \emph{trivial} action on complex projective space
and identifies the fixed points
	\[
	\mathbb{C}P^{\infty} = \left(\mathbb{H}P^{\infty}\right)^{G}.
	\]
The reader might compare this to the equivalence
$ \mathbb{R}P^{\infty}=\left(\mathbb{C}P^{\infty}\right)^{C_2}$
where $C_2$ acts by complex conjugation on projective space.
\end{remark}

We now study the loop spaces of $\mathbb{H}P^{\infty}$.

\begin{proposition}
$\Omega \mathbb{HP}^{\infty} \simeq S^{\lambda+1}$.
\end{proposition}
\begin{proof} The usual inclusion $S^{\lambda+2} \simeq \mathbb{H}P^1
\to \mathbb{H}P^{\infty}$ is equivariant for the action above
and induces a map $S^{\lambda+1} \to \Omega \mathbb{H}P^{\infty}$.
This is an underlying equivalence, and on fixed points for nontrivial
subgroups we have the standard map (see
Remark \ref{rmk:quat-fixed})
$S^1 \to \Omega \mathbb{C}P^{\infty}$, which is also an equivalence.
\end{proof}

\begin{remark}\label{rmk:nu-fixed-eta} As a corollary of the proof, we record that
the map $\nu: S^{2\lambda+3} \to S^{\lambda+2}$ becomes
$\eta: S^{3} \to S^2$ upon passage to fixed points.
\end{remark}

\begin{proposition}
If $G=C_2$, then $\Omega^{\sigma} \mathbb{H}P^{\infty} \simeq S^{\rho+1}$.
\end{proposition}
\begin{proof} Again, the usual inclusion
$S^{2\sigma+2} \to \mathbb{H}P^{\infty}$ is equivariant and we get
a map $S^{\sigma+2} = S^{\rho +1} \to \Omega^{\sigma}\mathbb{H}P^{\infty}$
which is an underlying equivalence. To compute the fixed points of
the right hand side, we use the fiber sequence of spaces
	\[
	\left(\Omega^{\sigma}\mathbb{H}P^{\infty}\right)^{C_2}
	\to \mathbb{C}P^{\infty} \to \mathbb{H}P^{\infty}
	\]
arising from the cofiber sequence
$C_{2+} \to S^0 \to S^{\sigma}$ of $C_2$-spaces.
But the fiber of $\mathrm{B}S^1 \to \mathrm{B}S^3$ is
$S^3/S^1 \simeq S^2$, by the Hopf fibration, which completes the proof.
\end{proof}

\begin{corollary}
If $G=C_2$, then
$$\Omega^{\lambda} S^{\lambda+1} \simeq \Omega^{\rho} S^{\rho+1}.$$
Moreover, this equivalence respects the inclusion of $S^1$ up to homotopy.
\end{corollary}

\begin{proof}
There is a chain of equivalences
$$\Omega^{\lambda} S^{\lambda+1} \simeq \Omega^{2 \sigma} \Omega \mathbb{H}P^{\infty} \simeq \Omega^{\rho} \Omega^{\sigma} \mathbb{H}P^{\infty} \simeq \Omega^{\rho} S^{\rho+1}.$$
\end{proof}

\begin{remark} If $V$ is a representation of $C_2$, then the free $\mathbb{E}_V$-algebra on any $C_2$-connected, pointed $C_2$-space $X$ is given by $\Omega^{V} \Sigma^{V} X$ \cite[Theorem 1]{rs}.  The above discussion therefore implies that the free $\mathbb{E}_{2\sigma}$-space and free $\mathbb{E}_{\rho}$-space on the pointed $C_2$-space $S^1$ coincide. A prototype of this result is that the
free group-like $\mathbb{E}_{\sigma}$-space and free group-like
$\mathbb{E}_{1}$-space on the pointed $C_2$-space
$S^0$ also coincide, i.e. 
$\Omega S^1 \simeq \Omega^{\sigma}S^{\sigma} \simeq \mathbb{Z}$
(with trivial action).
This can be proved in much the same way, using the equivalences
	\[
	\Omega^{\rho}\mathbb{C}P^{\infty} \simeq \mathbb{Z},
	\quad \Omega \mathbb{C}P^{\infty} \simeq S^{\sigma},
	\quad \text{and }\Omega^{\sigma}\mathbb{C}P^{\infty}
	\simeq S^1,
	\]
where $C_2$ acts on $\mathbb{C}P^{\infty}$ by complex
conjugation.
\end{remark}

As a consequence of the equivalence $\Omega \mathbb{H}P^{\infty}
\simeq S^{\lambda+1}$, we observe that $\Omega^{\lambda}S^{\lambda+1}$
inherits an $\mathbb{E}_{\lambda+1}$-structure.
This gives the fixed points $\left(\Omega^{\lambda}S^{\lambda+1}\right)^G$
an $\mathbb{E}_1$-algebra structure.

\begin{proposition} As an $\mathbb{E}_1$-space,
	\[
	\left(\Omega^{\lambda+1}\mathbb{H}P^{\infty}\right)^{G}
	\simeq \Omega^2S^3 \times \Omega S^2.
	\]
\end{proposition}

\begin{warning} At odd primes, the fixed points
$\mu^G$ of the map in Theorem \ref{thm:MainThom} is not an $\mathbb{E}_1$-map.
\end{warning}

\begin{proof}[Proof of the proposition]
Define $S^{\lambda/2}$ by the cofiber sequence
	\[
	G_+ \to S^0 \to S^{\lambda/2},
	\]
and notice that we have a cofiber sequence
	\[
	G_+ \wedge S^1 \to S^{\lambda/2} \to S^{\lambda}.
	\]
From the second cofiber sequence, we learn that there is a fiber sequence
	\[
	\Omega^{\lambda}\mathbb{H}P^{\infty}
	\to \mathrm{map}_*(S^{\lambda/2}, \mathbb{H}P^{\infty})
	\to
	\mathrm{map}_*(G_+ \wedge S^1, \mathbb{H}P^{\infty}).
	\]
Taking fixed points, we get
	\[
	\left(\Omega^{\lambda}\mathbb{H}P^{\infty}\right)^G
	\to \mathrm{map}_*(S^{\lambda/2}, \mathbb{H}P^{\infty})^G
	\to
	\Omega \mathbb{H}P^{\infty} \simeq S^3.
	\]
The first cofiber sequence identifies the middle term as the fiber
of the inclusion of fixed points:
	\[
	\mathrm{map}_*(S^{\lambda/2}, \mathbb{H}P^{\infty})^G
	\to \mathbb{C}P^{\infty} \to \mathbb{H}P^{\infty}.
	\]
In other words, $\mathrm{map}_*(S^{\lambda/2}, \mathbb{H}P^{\infty})^G
\simeq S^2$ and we can identify our previous fiber sequence with:
	\[
	\left(\Omega^{\lambda}\mathbb{H}P^{\infty}\right)^G
	\to S^2
	\to
	S^3.
	\]
The second map is null, so we learn that there is an equivalence:
	\[
	\left(\Omega^{\lambda}\mathbb{H}P^{\infty}\right)^G
	\simeq \Omega S^3 \times S^2,
	\]
and hence an equivalence of loop spaces:
	\[
	\left(\Omega^{\lambda+1}\mathbb{H}P^{\infty}\right)^G
	= \Omega \left(\Omega^{\lambda}\mathbb{H}P^{\infty}\right)^G
	\simeq \Omega(\Omega S^3 \times S^2).
	\]
\end{proof}

We stress that the above result does not concern
multiplicative structure on the Thom spectrum in question.
This is the subject of the next section at the prime two,
and of the subsequent section at odd primes.

%% file: Prime2.tex
\section{Extra structure at the prime \texorpdfstring{$2$}{2}}\label{sec:prime2}

Hopkins observed that the map $\mu: \Omega^2S^3 \to \mathrm{BO}$
admits a \emph{triple} delooping as the composite:
	\[
	\mathbb{H}P^{\infty} = \mathrm{BSp}(1) \longrightarrow
	\mathrm{BSp} \simeq \mathrm{B}^5\mathrm{O}
	\stackrel{\eta}{\longrightarrow} \mathrm{B}^4\mathrm{O}.
	\]
We would like to establish an equivariant version of this result.
The statement requires a few preliminaries.

The first results of this section hold for any finite subgroup $G \subseteq S^1$.  We will indicate later when we must restrict attention to $G=C_{2^n}$.

\begin{definition} A \textbf{$G$-$\mathbb{H}$-module}
is a real $G$-representation $V$ equipped with a 
$G$-equivariant algebra map $\mathbb{H} \to \mathrm{End}(V)$.
Here $\mathrm{End}(V)$ is the $G$-representation of 
all endomorphisms and we use the $G$-action
on $\mathbb{H}$ constructed in (\ref{cstr:quat-action}). 
More generally, 
a \textbf{$G$-$\mathbb{H}$-bundle} on a
$G$-space $X$ is a $G$-equivariant, real vector bundle
$E \to X$ together with a $G$-equivariant algebra
map $\mathbb{H} \to \mathrm{End}(E)$. 
\end{definition}

\begin{construction} For $F = \mathbb{R}$ or $\mathbb{H}$,
let $\mathcal{U}_{F}$ be a \textbf{complete $G$-$F$-universe}.
That is: $\mathcal{U}_F$ is a direct sum of infinitely many
finite dimensional $G$-$F$-modules which contains every
finite dimensional $G$-$F$-module as a summand (up
to isomorphism). Let $\mathrm{Gr}^F(\mathcal{U}_F)$
denote the infinite grassmanian with its induced $G$-action.
Then we define:
	\begin{align*}
	\mathrm{BO}_G=
	\mathrm{BGL}(\mathbb{R})&:= \mathrm{Gr}^{\mathbb{R}}
	(\mathcal{U}_{\mathbb{R}}),\\
	\mathrm{BGL}(\mathbb{H})&:= \mathrm{Gr}^{\mathbb{H}}
	(\mathcal{U}_{\mathbb{H}}).
	\end{align*}
\end{construction}

The $G$-space $\mathrm{BO}_G$ is well-known, and there
is an equivalence
	\[
	\Omega^{\infty}\mathrm{KO}_G = \mathbb{Z}
	\times \mathrm{BO}_G.
	\]
	
\begin{warning} Since we have
$\mathrm{BO}_G \simeq \mathrm{BGL}(\mathbb{R})$
and the former is well-known, we will use the notation
$\mathrm{BO}_G$. Beware, however, that
the $G$-space $\mathrm{BGL}(\mathbb{H})$
is \emph{not} the same as the space
$\mathrm{BSp}_G$ associated to equivariant, symplectic $K$-theory.
The latter does not incorporate a nontrivial action of $G$ on
$\mathbb{H}$. 
\end{warning}

\begin{warning} Neither $\mathrm{BO}_G$ nor
$\mathrm{BGL}(\mathbb{H})$ are equivariantly connected
when $G$ is nontrivial. For example, $\pi_0^G\mathrm{BO}_G$
is the group of virtual real representations of virtual dimension zero.
\end{warning}

\begin{theorem}[Karoubi] 
	\[
	\Omega^{\infty}\Sigma^{\lambda+2}\mathrm{KO}_G
	\simeq \mathbb{Z} \times \mathrm{BGL}(\mathbb{H})
	\]
where $\mathrm{BGL}(\mathbb{H})$ is the $G$-space
constructed above.
\end{theorem}
\begin{proof}[Proof sketch] We indicate how to recover this
result from the much more general work of Karoubi.
First, if we endow $\lambda$ with the standard negative
definite quadratic form, then the Clifford algebra
$C\ell(\lambda)$ is $G$-equivariantly isomorphic
to $\mathbb{H}$ as an algebra.
It follows from the `fundamental theorem'
\cite[Theorem 1.1]{karoubi-eqvt} that, when $X$
is compact,
$\mathrm{KO}_G^{2+\lambda}(X)$
is given by Karoubi's (graded) 
$K$-theory of the graded Banach
category of $G$-$C\ell(2+\lambda)$-bundles,
in the sense of \cite[Definition 2.1.6]{karoubi-gen}. 
From the interpretation of
this $K$-theory group explained on \cite[p.192]{karoubi-eqvt},
we learn that a class in 
$\mathrm{KO}_G^{2+\lambda}(X)$ is specified
by 
a $G$-$C\ell(2+\lambda)$-bundle $E$ on $X$ together with two extensions
to a $G$-$C\ell(3+\lambda)$-bundle structure on $E$.
Such a triple is declared trivial
if the two extensions give isomorphic bundles, and
two triples are equivalent if they become isomorphic after adding
a trivial triple.

The naturality statement in \cite[Thm. 3.10]{karoubi-book}
produces equivariant isomorphisms:
	\begin{align*}
	C\ell(2+\lambda) &\simeq \mathrm{M}_2(\mathbb{H})\\
	C\ell(3+\lambda) &\simeq \mathrm{M}_2(\mathbb{H})
	\times \mathrm{M}_2(\mathbb{H})
	\end{align*}

By Morita invariance, we may reinterpret elements in
$\mathrm{KO}_G^{2+\lambda}(X)$ as equivalence classes of
$G$-$\mathbb{H}$-bundles $E$ equipped
with two decompositions
$\eta_1: E \simeq E_0 \oplus E_1$ and $\eta_2: E \simeq E'_0 \oplus E'_1$.
Arguing
as in \cite[Proposition 4.26]{karoubi-book} and
\cite[Proposition 2.4]{segal-K}, one can show that every such
datum $(E, \eta_1, \eta_2)$  is equivalent to one of the
form $(X \times (M_0 \oplus M_1), \mathrm{id}, \eta)$ where
$M_0$ and $M_1$ are $G$-$\mathbb{H}$-modules.
After supplying a metric, we may replace $\eta$ by the data
of a sub-$G$-$\mathbb{H}$-module of $M_0 \oplus M_1$. 
For fixed $M_0$ and $M_1$, this data is equivalent
to an equivariant map 
$X \to \coprod_{k\ge 0}\mathrm{Gr}^{\mathbb{H}}_k(M_0\oplus M_1)$.
Now the result follows by the definition of a 
complete $G$-$\mathbb{H}$-universe and the construction
of $\mathrm{BGL}(\mathbb{H})$.
\end{proof}

At this point we may form the composite
	
	\[
	\xymatrix{
	\mathbb{H}P^{\infty} \ar[rr]^-{\mathcal{O}(-1)-1}&&
	\mathbb{Z} \times \mathrm{BGL}(\mathbb{H})
	\simeq \Omega^{\infty}\Sigma^{\lambda+2}\mathrm{KO}_G
	\ar[r]^-{\eta} &
	\Omega^{\infty}\Sigma^{\lambda+1}\mathrm{KO}_G
	}
	\]
where 
	\begin{itemize}
	\item $\mathcal{O}(-1)$ is the tautological
	$G$-$\mathbb{H}$-bundle on $\mathbb{H}P^{\infty}$.
	\item $\eta \in \pi_1^GKO_G$ is the image
	of $\eta \in \pi_1S^0 \subseteq \pi_1^GS^0$.
	\end{itemize}

To complete the construction, we will need
an equivariant version of the $J$-homomorphism.
The $J$-homomorphism
and equivariant spherical fibrations have been studied previously
(e.g. \cite{segal,mcclure,waner1,waner2}) and
it is shown in \cite{cw2} and \cite{shim} that
the classifying space of equivariant stable spherical fibrations
is an equivariant infinite loop space.
For the reader's convenience, we prove this here, as well as
the corresponding notion and results regarding Picard spectra,
which provides the target for the $J$-homomorphism. 
The construction below is natural from the point
of view of \cite{barwick-et-al}, from whom we draw
inspiration.

\begin{construction}[Picard $G$-spectrum]
Bachmann-Hoyois \cite[\S 9.2]{bachmann-hoyois}
refined the Hill-Hopkins-Ravenel
norm construction to a product-preserving functor:
	\[
	\underline{\mathsf{Sp}}^G:
	\mathsf{A}^{\textit{eff}}(G) \longrightarrow \mathsf{CAlg}(\mathsf{Cat}_{\infty}),
	G/H \mapsto \mathsf{Sp}^H,
	\]
where the left-hand side denotes the (effective) Burnside (2,1)-category
of finite $G$-sets and spans \cite{barwick}. 

Given any symmetric monoidal $\infty$-category $\mathcal{C}$,
define $\mathcal{P}\mathrm{ic}(\mathcal{C})$ to be the maximal subgroupoid
of objects which are invertible under the tensor product. This
is a group-like $\mathbb{E}_{\infty}$-space and so deloops
to a spectrum $\mathrm{pic}(\mathcal{C})$; moreover
the formation
of Picard spectra 
is product preserving \cite[2.2]{akhil-vesna}.
We define $\underline{\mathrm{pic}}(S^0)$
as the composite functor
	\[
	\underline{\mathrm{pic}}(S^0): \mathsf{A}^{\textit{eff}}(G) \longrightarrow
	\mathsf{CAlg}(\mathsf{Cat}_{\infty})
	\longrightarrow \mathsf{Sp}.
	\]
This is a spectral Mackey functor, and the $\infty$-category
of spectral Mackey functors is equivalent to the 
$\infty$-category of genuine $G$-spectra \cite{guillou-may,denis}.
Thus we have produced a $G$-spectrum 
which we call the \textbf{Picard $G$-spectrum}
of $S^0$. We denote by $\mathcal{P}\mathrm{ic}(S^0)$ the
$0$th space of this spectrum, which is a group-like
$G$-$\mathbb{E}_{\infty}$-space
(by which we mean, here and below, a $G$-commutative
monoid in the sense of \cite{denis}). We note that this $G$-space,
without extra structure,
may be obtained directly from $\underline{\mathsf{Sp}}^G$
by assigning to the orbit $G/H$ the maximal
subgroupoid of the full subcategory of $\mathsf{Sp}^H$
consisting of invertible objects. If one further restricts to the
full subcategory consisting of objects equivalent to $S^0$,
this defines the $G$-space $\mathrm{BGL}_1(S^0)$. This subcategory
is closed under the formation of norms and smash products and
hence inherits the structure of a group-like $G$-$\mathbb{E}_{\infty}$-space
for which the inclusion
	\[
	\mathrm{BGL}_1(S^0) \to \mathcal{P}\mathrm{ic}(S^0)
	\]
is a map of $G$-$\mathbb{E}_{\infty}$-spaces. In particular, 
the $G$-space $\mathrm{BGL}_1(S^0)$ is the zero space of
a $G$-spectrum $\Sigma\mathrm{gl}_1(S^0)$.

More generally, given any virtual $G$-representation $V$,
restricting, for each $H$,
to the full subcategory of objects equivalent to
$S^{\mathrm{res}_H(V)}$ produces a $G$-$\mathbb{E}_{\infty}$-space
canonically equivalent to $\mathrm{BGL}_1(S^0)$. 
\end{construction}

\begin{remark} The construction above works without
any change for the category of $p$-local $G$-spectra
to define $G$-infinite loop spaces $\mathcal{P}\mathrm{ic}(S^0_{(p)})$
and $\mathrm{BGL}_1(S^0_{(p)})$. 
\end{remark}

\begin{warning} The space $\mathcal{P}\mathrm{ic}(S^0)$
does not decompose into a disjoint union of copies of
$\mathrm{BGL}_1(S^0)$ when $G$ is nontrivial.
For example, take $G=C_2$ and let $Y \subseteq \mathcal{P}\mathrm{ic}(S^0)$
be the $C_2$-space with
	\begin{itemize}
	\item underlying space the subspace of nonequivariant spectra equivalent to $S^0$,
	\item fixed points the subspace of $C_2$-spectra
	equivalent to $S^V$ where $V$ has virtual dimension zero.
	\end{itemize}
Then $Y$ splits off of $\mathcal{P}\mathrm{ic}(S^0)$, but does not decompose further.
Indeed, $S^{\sigma-1}$ and $S^0$ lie in different components of $Y^{C_2}$
but restrict to the same component on the underlying space. On the other hand,
$\mathrm{BGL}_1(S^0)$ is $G$-connected, whence the claim. 
\end{warning}

\begin{construction}[Equivariant $J$-homomorphism] 
Let $\mathbf{Vect}_G$ denote the topological category
of finite-dimensional $G$-representations. We use
the same notation for the associated $\infty$-category.
Consider the product preserving functors
	\[
	\underline{\mathbf{Vect}}_G,
	\underline{\mathsf{Spaces}}_*^G:
	\mathsf{A}^{\textit{eff}}(G) \longrightarrow 
	\mathsf{CAlg}({\mathsf{Cat}_{\infty}})
	\]
given by:
	\begin{itemize}
	\item $\underline{\mathbf{Vect}}_G(G/H):= \mathbf{Vect}_H$
	with direct sum,
	functoriality by
	restriction and coinduction;
	\item $\underline{\mathsf{Spaces}}_*^G(G/H):=
	\mathsf{Spaces}_*^H = \mathsf{Psh}(\mathcal{O}_H, \mathsf{Spaces})$
	with $\wedge$ ,
	functoriality by restriction and norm defined by
		\[
		N_H^G(X):= \mathrm{map}_H(G, X)/\{f: * \in f(G)\}.
		\]
	\end{itemize}
The assignment $V \mapsto S^V$ produces a natural transformation
	\[
	\underline{\mathbf{Vect}}_G \to \underline{\mathsf{Spaces}}_*^G
	\stackrel{\Sigma^{\infty}}{\longrightarrow}
	\underline{\mathsf{Sp}}^G.
	\]
Restricting to maximal subgroupoids, 
and noting that each $S^V$ is invertible,
we get a natural transformation
	\[
	\underline{\mathbf{Vect}}_G^{\simeq}
	\longrightarrow \mathcal{P}\mathrm{ic}(S^0)
	\]
which we may regard as a map of $G$-$\mathbb{E}_{\infty}$-spaces.
The target is group-like, so this map factors through the group-completion
of the source. One can identify the underlying space of that
group-completion with $\mathbb{Z} \times \mathrm{BO}_G$,
so we have produced a $G$-$\mathbb{E}_{\infty}$-map
	\[
	J: \mathbb{Z} \times \mathrm{BO}_G
	\longrightarrow \mathcal{P}\mathrm{ic}(S^0).
	\]
We also denote by $J$ the restriction to $\{0\} \times \mathrm{BO}_G
= \mathrm{BO}_G$ as well as any deloopings.
\end{construction}

\begin{warning} Unlike the classical case, the restriction to virtual
dimension zero representations
	\[
	\mathrm{BO}_G \to \mathcal{P}\mathrm{ic}(S^0)
	\] 
does
\emph{not} factor through $\mathrm{BGL}_1(S^0)$ when
$G$ is nontrivial. Again, an explicit example is given by
the virtual representation $\sigma - 1$ when $G = C_2$.
\end{warning}

\begin{remark} Since $\Omega^{\lambda+1}\mathbb{H}P^{\infty}
=\Omega^{\lambda}S^{\lambda+1}$ is equivariantly connected, 
the map $\mathbb{H}P^{\infty} \to \Omega^{\infty}\Sigma^{\lambda+1}
\mathrm{KO}_G$ constructed above
factors through $\mathrm{B}^{\lambda+1}\mathrm{BO}_G$. 
\end{remark}

Now we specialize to the case $G = C_{2^n}$.

\begin{proposition} Let $g$ denote the composite
	\[
	\mathbb{H}P^{\infty} \longrightarrow 
	\mathrm{B}^{\lambda+1}\mathrm{BO}_G
	\longrightarrow
	\mathrm{B}^{\lambda+1}\mathcal{P}\mathrm{ic}(S^0)
	\]
Then $\Omega^{\lambda+1}g$ factors through $\mathrm{BGL}_1(S^0)$
and is homotopic to $\mu$ under the equivalence
${\Omega^{\lambda+1}\mathbb{H}P^{\infty} \simeq
\Omega^{\lambda}S^{\lambda+1}}$.
\end{proposition}
\begin{proof} Since $\Omega^{\lambda+1}\mathbb{H}P^{\infty}$
is equivariantly connected, $\Omega^{\lambda+1}g$
automatically factors through
$\mathrm{BGL}_1(S^0)$. To complete the proof, we need only
identify the map
	\[
	S^{\lambda+2} \to \mathbb{H}P^{\infty}
	\to \mathrm{B}^{\lambda+1}\mathcal{P}\mathrm{ic}(S^0).
	\]
To begin, notice that the map
	\[
	S^{\lambda+2} \to \mathrm{B}^{\lambda+1}\mathrm{BO}_G
	\]
corresponds to an element of $\mathrm{KO}_G^{-1}$. By
\cite[p.17]{atiyah-segal}, this group is $RO(G)/R(G)$.
For $G = C_{2^n}$ we have
	\[
	RO(G)/R(G) =\mathbb{F}_2\{1, \sigma\}.
	\]
The bundle we started with was restricted
from a bundle defined for $C_{2^{n+1}}$,
which excludes $\sigma$, so the element in question
is either 0 or 1. But we know that the
\emph{underlying} class is nonzero, so we must be looking
at the element $1$ in $RO(G)/R(G)$. Moreover, this
class corresponds precisely to the M\"obius bundle on $S^1$,
whence the claim.
\end{proof}

\begin{proof}[Proof of Theorem \ref{thm:main-even}
assuming Theorem \ref{thm:MainThom}]
Combine the above proposition with Theorem \ref{thm:MainThom}.
\end{proof}

%% file: HSpaceOrientation.tex
\section{An equivariant \texorpdfstring{$H$}{H}-space orientation}\label{sec:exotic}

The purpose of this section is to prove the following theorem.

\begin{theorem}\label{thm:h-orient} There is an $\mathbb{A}_2$-structure on
$S^{\lambda+1}_{(p)}$ such that
	\begin{enumerate}[\normalfont (i)]
	\item The adjoint to $1-p \in \pi_0(S^0_{(p)})^{\times}$ refines to
	an $\mathbb{A}_2$-map
		\[
	S_{(p)}^{\lambda+1} 
	\to \mathrm{B}^{\lambda+1}\mathrm{GL}_1(S^0_{(p)}).
		\]
	\item The composite
		\[
		S^{\lambda+1}_{(p)} \to \mathrm{B}^{\lambda+1}
		\mathrm{GL}_1(S^0_{(p)}) \to 
		\mathrm{B}^{\lambda+1}
		\mathrm{GL}_1(\mathrm{H}\underline{\mathbb{F}}_p)
		\]
	is nullhomotopic through $\mathbb{A}_2$-maps.
	\end{enumerate}
\end{theorem}

Before doing so, we record a corollary which follows from
\cite[X.6.4]{LMS} and the previous theorem:

\begin{corollary} The Thom class
	\[
	\left(\Omega^{\lambda}S^{\lambda+1}_{(p)}\right)^{\mu}
	\longrightarrow
	\mathrm{H}\underline{\mathbb{F}}_p
	\]
has the structure of a map of $\mathbb{A}_2$-algebras in
$\mathbb{E}_{\lambda}$-algebras.
\end{corollary}

\begin{remark}\label{rmk:a2-str} The $\mathbb{A}_2$-structure we define is necessarily
$p$-local. We only use this structure to produce a multiplication
on the homology of $\left(\Omega^{\lambda}S^{\lambda+1}_{(p)}\right)^{\mu}$
which is preserved by the map induced by the Thom class. The
localization map
	\[
	\Omega^{\lambda}S^{\lambda+1}\longrightarrow
	\Omega^{\lambda}S^{\lambda+1}_{(p)}
	\]
induces an isomorphism on mod $p$ homology and using this
one can show that the map
	\[
	\left(\Omega^{\lambda}S^{\lambda+1}\right)^{\mu}
	\longrightarrow
	\left(\Omega^{\lambda}S^{\lambda+1}_{(p)}\right)^{\mu}
	\]
is a $p$-local equivalence. On the other hand, 
the left hand side is automatically $p$-local, being
an (equivariant)
homotopy colimit of $p$-local spectra by construction.
We can transport
the $\mathbb{A}_2$-structure from the target to the source
and get a map
	\[
	\left(\Omega^{\lambda}S^{\lambda+1}\right)^{\mu} \to 
	\mathrm{H}\underline{\mathbb{F}}_p
	\] 
of $\mathbb{A}_2$-algebras
in $\mathbb{E}_{\lambda}$, even though the
$\mathbb{A}_2$-structure does not arise from applying
the Thom construction to an $\mathbb{A}_2$-map of $G$-spaces.
\end{remark}

We will need to recall a few facts about $\mathbb{A}_2$-monoid
objects, mainly to establish notation.

\begin{definition} Let $\mathcal{C}$
be an $\infty$-category which admits products.
For $0\le k \le \infty$,
an \textbf{$\mathbb{A}_k$-monoid} $X$ in $\mathcal{C}$
is a truncated simplicial object
	\[
	\mathbf{B}_{\le k}X:
	\Delta_{\le k}^{op} \longrightarrow \mathcal{C}
	\]
such that
	\begin{itemize}
	\item The object $(\mathbf{B}_{\le k}X)_0$ is final.
	\item If $k\ge 1$, then for $1\le j \le k$, the maps
		\[
		(\mathbf{B}_{\le k}X)_j \to (\mathbf{B}_{\le k}X)_1,
		\]
	induced by $\{i, i+1\} \to [j]$, exhibit the source
	as the $j$-fold product of $(\mathbf{B}_{\le k}X)_1$.
	\end{itemize}
In this case we denote $(\mathbf{B}_{\le k}X)_1$ by $X$.
If $\mathcal{C}$ admits
homotopy colimits, we denote the homotopy colimit of the diagram
$\mathbf{B}_{\le k}X$ by $\mathrm{B}_{\le k}X$.

The $\infty$-category of $\mathbb{A}_k$-monoids,
$\mathsf{Mon}_{\mathbb{A}_k}(\mathcal{C})$, is
the full subcategory of $\mathsf{Fun}(\Delta_{\le k}^{op}, \mathcal{C})$
spanned by the $\mathbb{A}_k$-monoids. Note that
restriction defines forgetful functors
	\[
	\mathsf{Mon}_{\mathbb{A}_k}(\mathcal{C})
	\to \mathsf{Mon}_{\mathbb{A}_j}(\mathcal{C})
	\]
for $j\le k$, and we have natural maps
	\[
	\mathrm{B}_{\le j}X \longrightarrow \mathrm{B}_{\le k}X.
	\]
\end{definition}

\begin{example}[$k=0$] An $\mathbb{A}_0$-monoid is a final
object of $\mathcal{C}$. The natural maps above provide
each $\mathrm{B}_{\le j}X$ with a basepoint.
\end{example}

\begin{example}[$k=1$] An $\mathbb{A}_1$-monoid in $\mathcal{C}$
is specified by the data of an object $X$ and a map
$\ast \to X$, where $\ast$ is a final object in $\mathcal{C}$.
The object $\mathrm{B}_{\le 1}X$ is computed as the colimit of
	\[
	\xymatrix{
	X \ar@<-.5ex>[r]\ar@<.5ex>[r] & \ast\ar[l]
	}
	\]
which is the suspension $\Sigma X$. 
\end{example}

\begin{example}[$k=2$] Suppose $\mathcal{C}$ admits
limits and colimits.
By \cite[A.2.9.16]{HTT}, extending
a diagram $\mathbf{B}_{\le 1}X$
to a diagram $\mathbf{B}_{\le 2}X$
is equivalent to specifying a factorization 
	\[
	\xymatrix{
	&C\ar[dr]^{(d_0, d_1, d_2)}&\\
	X \vee X  \ar[rr]
	\ar[ur] && X \times X \times X
	}
	\]
where $C \in \mathcal{C}$ is some object and 
$X \vee X$ denotes the pushout $X \amalg_{\ast} X$
where $\ast = (\mathbf{B}_{\le 1}X)_0$ is a final object.
In order for this extended diagram to define an $\mathbb{A}_{2}$-monoid,
the maps $d_0$ and $d_2$ must give an equivalence
$C \simeq X \times X$. Under this equivalence, the map
$X \vee X \to X \times X$ is the standard one. The only additional
data is the map $d_1: X \times X \simeq C \to X$.
In summary: an $\mathbb{A}_2$-monoid in $\mathcal{C}$
is precisely the data of a pointed object $X$
together with a map $m: X \times X \to X$ which extends
the fold map $X \vee X \to X$. It follows that the cofiber of
$\mathrm{B}_{\le 1}X \to \mathrm{B}_{\le 2}X$ is
given by $\Sigma^2(X \wedge X)$. 
\end{example}

\begin{remark} If we endow $\mathcal{C}$ with the Cartesian monoidal
structure, then we see that an $\mathbb{A}_2$-monoid object in
$\mathcal{C}$ is the same data as an $\mathbb{A}_2$-algebra object
as in Definition \ref{dfn:a2}.
\end{remark}

\begin{example}[Loop spaces]\label{ex:loops} If $\mathcal{C} = \mathsf{Spaces}^G$
and $Y$ is a pointed $G$-space, then $\Omega Y$ has
a natural $\mathbb{A}_{\infty}$-structure.
\end{example}

Now we will restrict attention to the $\infty$-category
$\mathsf{Spaces}^{G}$ of $G$-spaces.
The following is proved just as in the classical case,
for which there are many references. The earliest
appears to be \cite[Proposition 3.5]{stash-ab}.

\begin{lemma} If $X$ is
an $\mathbb{A}_2$-algebra, then
the map $X \to \Omega \mathrm{B}_{\le 2}X$
adjoint to
	\[
	\Sigma X = \mathrm{B}_{\le 1}X
	\to \mathrm{B}_{\le 2} X
	\]
extends to an $\mathbb{A}_2$-map.
\end{lemma}

We now return to the case of interest. We begin by establishing
the existence of the $\mathbb{A}_2$-structure we need on
$S^{\lambda+1}$.

\begin{proposition}\label{prop:exotic-h} For $p$ an odd prime,
there is an 
$\mathbb{A}_2$-structure on $S_{(p)}^{\lambda+1}$
with the property that the map $\Sigma X \to \mathrm{B}_{\le 2}X$
stably splits.
\end{proposition}

We will deduce this proposition from the following calculation,
which is an equivariant version of a classical result
(see, e.g. \cite{james}). Here we use
	\[
	E: [X,Y] \to [\Sigma X, \Sigma Y]
	\]
to denote the homomorphism on
equivariant homotopy classes of maps
given by suspending by
$S^1$ equipped with the trivial $G$-action.

\begin{proposition}\label{prop:dbl-susp}
Denote by $\nu: S^{2\lambda+3} \to S^{\lambda+2}$
the attaching map for the inclusion
$S^{\lambda+2}\simeq \mathbb{H}P^1 \to
\mathbb{H}P^{2}$.
Then, after localization at $p$,
	\[
	E(2\nu) \in E^2(\pi_{2\lambda+2}S^{\lambda+1}).
	\]
\end{proposition}

\begin{proof}[Proof of Proposition \ref{prop:exotic-h} assuming
Proposition \ref{prop:dbl-susp}] 
Recall from 
that there is an equivalence $\Omega \mathbb{H}P^{\infty} \simeq
S^{\lambda+1}$ and hence, by Example \ref{ex:loops},
we get an $\mathbb{A}_{\infty}$-structure on
$S^{\lambda+1}$.
This, in turn, determines an $\mathbb{A}_2$-structure on
$S^{\lambda+1}$, and
the attaching map for $\mathrm{B}_{\le 1}S^{\lambda+1}
\to \mathrm{B}_{\le 2}S^{\lambda+1}$ is precisely
	\[
	\nu: S^{2\lambda+3} \to S^{\lambda+2}.
	\]

In general,
if one modifies an $\mathbb{A}_2$-structure on $X$ by an element
$d \in [X \wedge X, X]$, then the attaching map
	\[
	\Sigma (X \wedge X) \to \Sigma X
	\]
for $\mathrm{B}_{\le 2}X$ is altered by $E(d)$. 
After inverting 2, Proposition \ref{prop:dbl-susp} implies that
$E(\nu) = E^2(x)$ for some $x \in \pi_{2\lambda+2}S^{\lambda+1}$. 
So alter the $\mathbb{A}_2$-structure above by $x$ and the suspension
of the attaching map for $\mathrm{B}_{\le 2}S^{\lambda+1}$ becomes
null, proving the result.
\end{proof}

Now we turn to the proof of Proposition \ref{prop:dbl-susp}. 
We will deduce this theorem from a slightly stronger result.
Recall that, given classes $x \in [\Sigma A, X]$ and
$y \in [\Sigma B, X]$, the Whitehead product
${[x,y] \in [\Sigma (A \wedge B), X]}$ is induced by
the commutator of $\pi_Ax$ and $\pi_By$ in the group
$[\Sigma(A\times B), X]$. 

\begin{lemma}\label{lem:bracket}
Let $\iota_{\lambda+2} \in \pi_{2\lambda+3}S^{\lambda+2}$
be the fundamental class. Then, after localization at $p$,
	\[
	[\iota_{\lambda+2}, \iota_{\lambda+2}] \equiv 2\nu 
	\mod
	E(\pi_{\lambda+2}S^{\lambda+1}).
	\]
\end{lemma}

\begin{proof}[Proof of Proposition \ref{prop:dbl-susp} assuming
Lemma \ref{lem:bracket}.] Lemma \ref{lem:bracket} states
$2\nu - [\iota_{\lambda+2}, \iota_{\lambda+2}]$ lies in
the image of $E$. Hence $E(2\nu) - E([\iota_{\lambda+2},
\iota_{\lambda+2}])$ lies in the image of $E^2$.
The result now follows from the observation that suspensions
of Whitehead products vanish. Indeed, with notation
as in the definition of the Whitehead product above,
$E([x,y])$ is computed as a commutator in
$[\Sigma^2(A \times B), X]$. But this is an abelian group
by the Eckmann-Hilton argument, so commutators vanish.
\end{proof}

In order to prove Lemma \ref{lem:bracket} we will establish
an exact sequence of the form
	\[
	\xymatrix{
	\pi_{2\lambda+2}S^{\lambda+1}\ar[r]^{E}&
	\pi_{2\lambda+3}S^{\lambda+2}\ar[r]^{H} &
	\pi_{2\lambda+3}S^{2\lambda+3}
	}
	\]
and then identify the image of the Whitehead product
in the last group. To that end, we note that the James
splitting
	\[
	\Sigma \Omega \Sigma X \simeq
	\Sigma \left(\bigvee_{k\ge 1} X^{\wedge k}\right)
	\]
holds in $\mathsf{Spaces}^G_*$ (see \cite{kron}). This provides
a natural transformation
	\[
	H: \Omega\Sigma X \to \Omega \Sigma X^{\wedge 2}
	\]
which induces a map
	\[
	H: \pi_{\star+1} \Sigma X \to \pi_{\star+1}\Sigma X^{\wedge 2}
	\]
for any $X$.

\begin{lemma} The sequence
	\[
	\xymatrix{
	S^{\lambda+1} \ar[r]^{E} & \Omega S^{\lambda+2}\ar[r]^H & 
	\Omega S^{2\lambda+3}
	}
	\]
is a fiber sequence when localized at $p$.
\end{lemma}
\begin{proof} Let $F$ denote the homotopy fiber of
$H$ so that we have a natural map $S^{\lambda+1}\to F$. 
We would like to show this is an equivalence.
Since restriction to underlying spaces and fixed points
preserves homotopy limits and colimits, we are reduced to
the nonequivariant statement that
	\[
	\xymatrix{
	S^{2n+1} \ar[r]^E& \Omega S^{2n+2} \ar[r]^H& \Omega S^{4n+3}
	}
	\]
is a fiber sequence when localized
at $p$ for $n=0, 1$. In fact, it is a classical result of
James \cite{james}
that this is a $p$-local fiber sequence for any $n\ge 0$.
\end{proof}

We will need some control over the last term in this sequence,
which is provided by an equivariant version of the 
Brouwer-Hopf degree theorem. For us, the only fact we need
is that the homomorphism
	\[
	\pi_{2\lambda+3}S^{2\lambda+3} \to \bigoplus_{K \subseteq G} \mathbb{Z},
	\]
recording each of the degrees of a map on $K$-fixed points, is
an injection. See, e.g., \cite[8.4.1]{tdieck}. We now prove the only remaining lemma necessary for
producing the exotic $H$-space structure on $S^{\lambda+1}$.

\begin{proof}[Proof of lemma \ref{lem:bracket}] The formation
of Whitehead products commutes with passage to fixed points
and restriction to underlying classes, as does the map $H$.
From the remarks above, it then suffices to check the nonequivariant
formulas:
	\[
	H([\iota_4, \iota_4]) = 2H(\nu),
	\]
	\[
	H([\iota_2, \iota_2]) = 2H(\nu^K), \quad K\ne\{e\}.
	\]
But $\nu$ and $\nu^K = \eta$ 
(see Remark \ref{rmk:nu-fixed-eta}) have Hopf invariant 1,
while $[\iota_{2n}, \iota_{2n}]$ has Hopf invariant 2 for any $n\ge 1$,
whence the result.
\end{proof}

Since the attaching map in $\mathrm{B}_{\le 2}S^{\lambda+1}_{(p)}$
is stably null, the following lemma is immediate.

\begin{lemma}\label{lem:h-extend-p}
There exists a dotted map making the diagram
below commute up to homotopy in $\mathsf{Sp}^G$:
	\[
	\xymatrix{
	S_{(p)}^{\lambda+2} \ar[r]^-{1-p} \ar[d]& 
	\Sigma^{\lambda+2}\mathrm{gl}_1S^0_{(p)}\\
	\Sigma^{\infty}\mathrm{B}_{\le 2}S_{(p)}^{\lambda+1}\ar@{-->}[ur] &
	}
	\]
\end{lemma}

We will eventually need to produce a Thom isomorphism
in mod $p$ cohomology which respects our extra structure. For
that we require the next lemma.

\begin{lemma}\label{lem:h-null}
Choose a dotted map $\tilde{f}$ as in the previous
lemma. Then the composite
	\[
	\xymatrix{
	\Sigma^{\infty}\mathrm{B}_{\le 2} S_{(p)}^{\lambda+1}\ar[r]^{\tilde{f}}&
	\Sigma^{\lambda+2}\mathrm{gl}_1(S^0_{(p)}) \ar[r] &
	\Sigma^{\lambda+2}\mathrm{gl}_1(\mathrm{H}\underline{\mathbb{F}}_p)
	}
	\]
is null.
\end{lemma}
\begin{proof} The composite
	\[
	\xymatrix{
	S^{\lambda+2}_{(p)} \ar[r] &
	\Sigma^{\infty}\mathrm{B}_{\le 2} S_{(p)}^{\lambda+1}\ar[r]^{\tilde{f}}&
	\Sigma^{\lambda+2}\mathrm{gl}_1(S^0_{(p)}) \ar[r] &
	\Sigma^{\lambda+2}\mathrm{gl}_1(\mathrm{H}\underline{\mathbb{F}}_p)
	}
	\]
vanishes since 
$1-p = 1 \in \pi_0^G(\mathrm{gl}_1(\mathrm{H}\underline{\mathbb{F}}_p)$
is the basepoint component. So the map $\tilde{f}$ factors through
some map
	\[
	S^{2\lambda+4} \longrightarrow 
	\Sigma^{\lambda+2}\mathrm{gl}_1(\mathrm{H}\underline{\mathbb{F}}_p).
	\]
But
	\[
	\pi_{\lambda+2}^{G}\mathrm{gl}_1(\mathrm{H}\underline{\mathbb{F}}_p)
	\simeq \pi_{\lambda+2}^{G}\mathrm{H}\underline{\mathbb{F}}_p
	= 0
	\]
since $S^{\lambda+2}$ is $2$-connective, whence the claim.
\end{proof}

Finally, we arrive at the proof of the main theorem of the section.

\begin{proof}[Proof of Theorem \ref{thm:h-orient}]
Choose a dotted map as in Lemma
\ref{lem:h-extend-p} and let $f$ be its adjoint,
	\[
	f: \mathrm{B}_{\le 2}S^{\lambda+1}_{(p)}
	\longrightarrow \mathrm{B}^{\lambda+2}\mathrm{GL}_1(S^0_{(p)}).
	\]

Then the map $1-p: S^{\lambda+1}_{(p)} \to 
\mathrm{B}^{\lambda+1}\mathrm{GL}_1(S^0_{(p)})$
factors as a composite:
	\[
	\xymatrix{
	S^{\lambda+1}_{(p)}
	\ar[r] & \Omega \mathrm{B}_{\le 2}S^{\lambda+1}_{(p)}
	\ar[r]^-{\Omega f} &
	\Omega \mathrm{B}^{\lambda+2}\mathrm{GL}_1(S^0_{(p)})
	\ar[r] & \mathrm{B}^{\lambda+1}\mathrm{GL}_1(S^0_{(p)})
	},
	\]
each of which is an $H$-map. This proves part (i) of the theorem.

To prove part (ii), consider the diagram:
	\[
	\xymatrix{
	S^{\lambda+1}_{(p)}
	\ar[r] & \Omega \mathrm{B}_{\le 2}S^{\lambda+1}_{(p)}
	\ar[r]^-{\Omega f} &
	\Omega \mathrm{B}^{\lambda+2}\mathrm{GL}_1(S^0_{(p)})
	\ar[d]_-{\Omega\mathrm{B}^{\lambda+2}\mathrm{GL}_1(\iota)}
	\ar[r] & \mathrm{B}^{\lambda+1}\mathrm{GL}_1(S^0_{(p)})\ar[d]\\
	&& 
	\Omega\mathrm{B}^{\lambda+2}\mathrm{GL}_1(
	\mathrm{H}\underline{\mathbb{F}}_p) \ar[r] &
	\mathrm{B}^{\lambda+1}
	\mathrm{GL}_1(\mathrm{H}\underline{\mathbb{F}}_p)
	}
	\]
where $\iota: S^0_{(p)} \to \mathrm{H}\underline{\mathbb{F}}_p$
is the unit map.

The composite 
	\[
	\xymatrix{
	\mathrm{B}_{\le 2}S^{\lambda+1}_{(p)}
	\ar[r]^-{f} & \mathrm{B}^{\lambda+2}\mathrm{GL}_1(S^0_{(p)})
	\ar[rr]^-{\mathrm{B}^{\lambda+2}\mathrm{GL}_1(\iota)} &&
	\mathrm{B}^{\lambda+2}\mathrm{GL}_1(
	\mathrm{H}\underline{\mathbb{F}}_p)
	}
	\]
is null by Lemma \ref{lem:h-null}. The loop of this composite
is then null through $\mathbb{A}_{\infty}$-maps and the result follows.
\end{proof}

%% file: NormArgument.tex
\section{Computing the zeroth homotopy Mackey functor}\label{sec:norm}

In this section, we establish that the zeroth homotopy Mackey functor of our Thom spectrum is as expected.  That is to say, we give a proof that
	\[
	\underline{\pi}_0\left(\Omega^{\lambda}S^{\lambda+1}\right)^{\mu}
	=
	\underline{\mathbb{F}}_p.
	\]

By construction, 
$\left(\Omega^{\lambda}S^{\lambda+1}\right)^{\mu}$ receives
a map from the mod $p$ Moore spectrum
$M(p) = (S^1)^{\mu}$. This is enough to guarantee that
$p=0$ in 
$\underline{\pi}_0\left(\Omega^{\lambda}S^{\lambda+1}\right)^{\mu}$.
However, $\underline{\pi}_0S^0$ is the Burnside Mackey
functor $\underline{A}$, and $\underline{A}/(p)$ is not
$\underline{\mathbb{F}}_p$. For example, when $G=C_p$,
we have
	\[
	\underline{A}/(p) = \begin{gathered}
	\xymatrix{
	\mathbb{F}_p\left\{ [C_p]\right\}\ar@/_/[d]\\
	\mathbb{F}_p\ar@/_/[u]
	}
	\end{gathered}.
	\]
We will need to use some extra structure on 
$\left(\Omega^{\lambda}S^{\lambda+1}\right)^{\mu}$
to show that $[C_p]$ also vanishes. More generally,
we must show that $[C_{p^n}/C_{p^k}]$ vanishes
in the Hurewicz image for all $k$.

For the remainder of this section we write $G = C_{p^n}$
for a cyclic group of prime power order.

\begin{definition} We say that a $G$-spectrum
$X$ is \textbf{weakly normed} if it is equipped
with a map $S^0 \to X$, and, for each $H \subseteq G$, a map of $H$-spectra
$\mathrm{N}^HX \to X$ such that the diagram
	\[
	\xymatrix{
	\mathrm{N}^H(S^0)\ar[d]\ar@{=}[r] & S^0\ar[d]\\
	\mathrm{N}^H(X) \ar[r] & X
	}
	\]
commutes in the homotopy category.
\end{definition}

\begin{remark} This is the weakest structure necessary to run the
arguments below, but it is perhaps not the most natural definition.
In most examples one at least has compatibility between the norms as
$H$ varies, and the map $S^0 \to X$ acts as a unit for an
underlying multiplication.
\end{remark}

The Thom spectrum $\left(\Omega^{\lambda}S^{\lambda+1}\right)^{\mu}$
is weakly normed, as we now show. This result is well-known; compare,
for example, \cite[Theorem 2.12]{hill-disks}.

\begin{lemma}\label{lem:lambda-has-norms} If $X$ is an $\mathbb{E}_{\lambda}$-algebra
then it is canonically weakly normed.
\end{lemma}

\begin{proof}[Proof of \ref{lem:lambda-has-norms}] 
In order to conform with the existing literature
we will present a proof within the point-set model of orthogonal $G$-spectra
as in \cite[\S B]{HHR}. In particular, we will model $X$ by a 
positively cofibrant orthogonal $G$-spectrum.

Since the
restriction of an $\mathbb{E}_{\lambda}$-algebra
is still an $\mathbb{E}_{\lambda}$-algebra,
it will suffice to construct the norm
$\mathrm{N}^GX \to X$. 

By definition, $X$ comes equipped with a map
	\[
	\widetilde{\mathrm{Conf}}_{p^n}(\lambda)_+ \wedge_{\Sigma_{p^n}}
	X^{\wedge p^n} \longrightarrow X,
	\]
where $\widetilde{\mathrm{Conf}}_{p^n}(\lambda)$ denotes
the $G$-space of configurations of $p^n$ ordered points in
$\lambda$. Consider the inclusion
$G \hookrightarrow \Sigma_{p^n}$ which sends
a generator to the standard $p^n$-cycle
$(1, 2, ..., p^n)$ and let
$\Gamma$ denote the graph of this inclusion.
Let $\zeta = e^{2\pi i/p^n}$.
Then the ordered tuple $(1, \zeta, \zeta^2, ..., \zeta^{p^n-1})
\in \widetilde{\mathrm{Conf}}_{p^n}(\lambda)$ produces
a $G \times \Sigma_{p^n}$-equivariant inclusion:
	\[
	\frac{G \times \Sigma_{p^n}}{\Gamma}
	\longrightarrow \widetilde{\mathrm{Conf}}_{p^n}(\lambda).
	\]
This, in turn, gives us a map
	\[
	\left(\frac{G \times \Sigma_{p^n}}{\Gamma}\right)_+ \,\,
	\underset{\Sigma_{p^n}}{\wedge} X^{\wedge p^n}
	\longrightarrow X.
	\]
To complete the proof, we note that (\cite[Prop. 6.2]{blumberg-hill}),
for any $G$-spectrum $Y$,
we have
	\[
	\left(\frac{G \times \Sigma_{p^n}}{\Gamma}\right)_+\,\,
	\underset{\Sigma_{p^n}}{\wedge} Y^{\wedge p^n}
	\simeq \mathrm{N}^GY.
	\]
\end{proof}

\begin{proposition}\label{prop:norm-and-vanish} Suppose $X$ is 
weakly normed.
Suppose further
that $p=0 \in \pi_0^{G}X$. Then
$[H/K] = 0 \in \pi_0^{H}X$ for all
$K \subseteq H \subseteq G$.
\end{proposition}

\begin{corollary}\label{cor:pi-0} We have
	\[
	\underline{\pi}_0\left(\Omega^{\lambda}S^{\lambda+1}\right)^{\mu}
	= \underline{\mathbb{F}}_p.
	\]
\end{corollary}

\begin{proof}[Proof of Proposition \ref{prop:norm-and-vanish}]
Recall that $G = C_{p^n}$. If the result is proved for $C_{p^{n-1}}
\subseteq G$, then the classes
	\[
	p, \mathrm{tr}_{C_{p^{n-1}}}^G([C_{p^{n-1}}])
	= [G], \mathrm{tr}_{C_{p^{n-1}}}^G([C_{p^{n-1}}/C_{p}])
	= [G/C_p], ..., [G/C_{p^{n-2}}]
	\]
all vanish in $\pi_0^GX$. The result now follows from the next lemma.
\end{proof}

\begin{lemma}\label{lem:norm-p-formula} If $X$ is weakly normed, then
	\[
	\mathrm{N}^G(p) \equiv -[G/C_{p^{n-1}}] \mod 
	(p, [G/K] : K \subsetneq C_{p^{n-1}}).
	\]
\end{lemma}
\begin{proof} It suffices to prove this
formula when $X = S^0$, i.e. for the Burnside Mackey
functor $\underline{A}$.
By \cite[Lemma A.36]{HHR}, the
norm of $p$ is the class of the $G$-set
$\mathrm{map}(G, \{1, ..., p\})$. By recording the size of the
image of a map, we get an equality in $A(G)$:
	\[
	\left[\mathrm{map}(G, \{1, ..., p\})\right]
	= \sum_{0<k\le n} \binom{p}{k} [\mathrm{surj}(G, \{1, ..., k\})]
	\]
where $\mathrm{surj}(G, \{1, ..., k\})$ denotes the $G$-set of
surjective maps $G \to \{1, ..., k\}$. So we have
	\[
	\mathrm{N}^G(p) \equiv [\mathrm{surj}(G, \{1, ..., p\})] \mod p.
	\]
We are only concerned with the orbits in $\mathrm{surj}(G, \{1, ..., p\})$
with isotropy $C_{p^{n-1}}$ or $G$. There are $(p-1)!$ orbits with isotropy $C_{p^{n-1}}$, namely the orbit of the quotient map $G \to G/C_{p^{n-1}} \simeq \{1, ..., p\}$ and the orbits of the maps obtained from this one by reordering $\{2, ..., p\}$.
There are $p$ orbits with isotropy $G$, namely the $p$
constant maps. This completes the proof.
\end{proof}

%% file: NormArgumentII.tex
\section{Toward the first homotopy groups of the fixed points}\label{sec:normII}

The tom Dieck splitting \cite{tomdieck,segal},
	\[
	(S^0)^{C_{p^n}} =
	\bigoplus_{0\le k\le n} S^0_{h\mathrm{Aut}_{C_{p^n}}(C_{p^n}/C_{p^k})}
	\]
induces, upon taking $\pi_1$, a map
	\[
	\alpha:
	\pi_1^{C_{p^n}}(S^0) \to \mathrm{Aut}(C_{p^n}/C_{p^{n-1}})
	\cong \mathbb{Z}/p.
	\]
This section is devoted to a proof of the following proposition:

\begin{proposition}\label{find-x} Let $X$ denote the Thom spectrum $\left(\Omega^{\lambda} S^{\lambda+1}\right)^{\mu}$.
Then there is an element $x \in \pi^{C_{p^n}}_1 S^0$ such that:
	\begin{enumerate}[{\rm (i)}]
	\item $\alpha(x) \ne 0$
	\item $x$ is sent to zero under the unit map $\pi^{C_{p^n}}_1(S^0) \to \pi_1(X^{C_{p^n}})$
	\end{enumerate}
\end{proposition}

The key observation necessary to prove Proposition \ref{find-x} is the following:
\begin{lemma}\label{find-y} 
Let 
\[\widetilde{1+p}:	\Omega^{\lambda+1}S^{\lambda+1} \to 
	\mathrm{GL}_1(S^0_{(p)})\]
	denote the unique $\Omega^{\lambda+1}$ map extending $1+p \in \pi_0^{C_{p^n}}(\mathrm{GL}_1(S^0_{(p)}))$.
Then there is an element 
$y \in \pi^{C_{p^n}}_1(\Omega^{\lambda+1}S^{\lambda+1})$ whose image, $x$, under the map $\widetilde{1+p}$
has $\alpha(x)$ nonzero.
\end{lemma}

\begin{proof}[Proof of Proposition \ref{find-x} assuming Lemma \ref{find-y}]
First observe that the Thom spectrum associated to the map

	\[
	\mathrm{B}(\widetilde{1+p}):\Omega^{\lambda}S^{\lambda+1} \to \mathrm{BGL}_1(S^0),
	\]
is equivalent to $X$. By Lemma \ref{find-y}, we have a map
$S^2 \to \Omega^{\lambda}S^{\lambda+1}$ such that the induced map
on Thom spectra becomes
	\[
	S^0/x \to X.
	\]
In particular, the element $x \in \pi_1^{C_{p^n}}S^0$ maps to zero in $\pi^{C_{p^n}}_1X$,
which proves (ii).
\end{proof}

Now we turn to the proof of Lemma \ref{find-y}. 
Observe that
we have a commutative diagram
	\[
	\xymatrix{
	\mathrm{Free}_{\mathbb{E}_{\lambda+1}}(*) \ar[r]\ar[d] &
	\mathrm{Free}_{\mathbb{E}_{\infty \rho}}(*) \ar[d] & \\
	\Omega^{\lambda+1}S^{\lambda+1} \ar[r] & 
	\Omega^{\infty}S^0 \ar[r]^{\widetilde{1+p}} & \mathrm{GL}_1(S^0_{(p)})
	}
	\]
Here, $\mathbb{E}_{\infty\rho}$ denotes the union of the
operads of little disks in the representations
$m\rho$, where $\rho$ is the regular representation
(see \cite{hauschild, hill-disks}). We recall that the
$C_{p^n}$-space $\mathbb{E}_{\infty\rho}(k)/\Sigma_k$ is 
a model for $\mathrm{B}_{C_{p^n}}\Sigma_k$,
the classifying $C_{p^n}$-space for
principal $\Sigma_k$-bundles.

\begin{construction}
Consider the diagram
	\[
	\xymatrix{
	S(\lambda) \ar[r]\ar[d] & S^1 \times \mathbb{R}^{\lambda+1}\ar[dl]^{\mathrm{proj}_{S^1}}
	\\
	S^1
	}
	\]
where $S(\lambda) \to S^1$ is the quotient map and
$S(\lambda) \subseteq \mathbb{R}^{\lambda}
\subseteq \mathbb{R}^{\lambda+1}$ is the inclusion.
This gives an equivariant
$S^1$-family of $p^n$-points in $\mathbb{R}^{\lambda+1}$, and hence
is classified by a map
	\[
	S^1 \to \mathbb{E}_{\lambda+1}(p^n)/\Sigma_{p^n}.
	\]
This defines an element $y \in \pi_1\mathrm{Free}_{\mathbb{E}_{\lambda+1}}(*)$.

The image of $y$ in $\mathrm{Free}_{\mathbb{E}_{\infty \rho}}(*)$
corresponds to the map $S^1 \to \mathrm{B}_{C_{p^n}}\Sigma_{p^n}$
classifying the cover $S(\lambda) \to S^1$. We will denote this class
also by $y$. 
\end{construction}

Combining the above construction with the diagram:
	\[
	\xymatrix{
	&&\Omega^{\infty}S^0\ar[d]\\
	\mathrm{Free}_{\mathbb{E}_{\infty\rho}}(*) \ar[r]_{\widetilde{1+p}}
	 \ar[urr]^{\widetilde{1+p}}&\mathrm{GL}_1(S^0_{(p)})
	\ar[r] & \Omega^{\infty}S^0_{(p)},
	}
	\]
we see then that Lemma \ref{find-y} follows from:

\begin{lemma} \label{show-y-good} The image, $x$, of $y$ under the map
	\[
	\mathrm{Free}_{\mathbb{E}_{\infty \rho}}(*)
	\stackrel{\widetilde{1+p}}{\longrightarrow}
	\Omega^{\infty}S^0
	\]
has $\alpha(x)$ nonzero. (Here we are using the
\emph{multiplicative} monoid structure on $\Omega^{\infty}S^0$ to define
$\widetilde{1+p}$).
\end{lemma}

In order to prove Lemma \ref{show-y-good}, we will model the
map $\widetilde{1+p}$ using finite $C_{p^n}$-sets.
Let $\mathrm{Fin}^{C_{p^n}}$ denote the groupoid of
finite $C_{p^n}$-sets, so that we have an equivalence
	\[
	\mathrm{Fin}^{C_{p^n}} \simeq
	(\mathrm{Free}_{\mathbb{E}_{\infty\rho}}(*))^{C_{p^n}},
	\]
and the group completion of this space (under disjoint union)
gives $(\Omega^{\infty}S^0)^{C_{p^n}}$ \cite{segal}. Unwinding the definitions,
the map $\widetilde{1+p}$, on fixed points, arises from the functor
	\[
	\mathrm{Fin}^{C_{p^n}} \to \mathrm{Fin}^{C_{p^n}},
	J \mapsto \mathrm{map}(J, \{1, ..., p+1\}).
	\]

\begin{remark} The invariant $\alpha$ can be described
in terms of finite $C_{p^n}$-sets as follows. Given
a finite $C_{p^n}$-set $K$ and an automorphism $\theta$,
let $K'$ denote the summand with isotropy equal to $C_{p^{n-1}}$.
Then $\theta\vert_{K'}$ can be written as a permutation of
$K'/C_{p^n}$ followed by an automorphism of each individual
summand of $K'$. The composition of these automorphisms
is then an element in $\mathrm{Aut}(C_{p^n}/C_{p^{n-1}})\simeq \mathbb{Z}/p$.
This describes $\alpha((K, \theta))$ where we view
$(K,\theta)$ as specifying an element in $\pi_1(\Omega^{\infty}S^0)^{C_{p^n}}$.
\end{remark}

Notice that the class $y$ arose from the cover $S(\lambda) \to S^1$,
which has fibers isomorphic to $C_{p^n}$ and monodromy
given by the generator $\gamma \in C_{p^n}$. It follows 
from the previous remark that
Lemma \ref{show-y-good} is equivalent to the following:

\begin{lemma} $\alpha(\mathrm{map}(C_{p^n}, \{1, ..., p+1\}), \gamma^*)$
is nonzero. 
\end{lemma}
\begin{proof} We have a decomposition of pairs
	\[
	(\mathrm{map}(C_{p^n}, \{1, ..., p+1\}), \gamma^*)
	\simeq
	\coprod_{1\le k\le p+1} \binom{p+1}{k}(\mathrm{Surj}(C_{p^n}, 
	\{1, ..., k\}), \gamma^*)
	\]
The invariant $\alpha$ vanishes mod $p$,
so we need only consider the summands with
$k=1, p,$ and $p+1$. When $k=1$, the isotropy groups
are all of $C_{p^n}$, so $\alpha$ vanishes for these summands.
Notice that a surjection
$C_{p^n} \to \{1, ..., k\}$ has isotropy containing $C_{p^{n-1}}$
if and only if it factors through $C_{p^n}/C_{p^{n-1}}$;
this rules out the case $k=p+1$. For $k=p$, we are then
left with the $(p-1)!$ different orbits obtained from the orbit
of the canonical map
	\[
	C_{p^n} \to C_{p^n}/C_{p^{n-1}} \simeq \{1, ..., p\}
	\]
by reordering the elements $\{2, ..., p\}$. The value of $\alpha$
on this orbit is $-1$ (since $\gamma^*$ is the automorphism
given by precomposition
with $\gamma$ on $\mathrm{map}(C_{p^n}, \{1, ..., p+1\})$,
while the action is given by precomposing with $\gamma^{-1}$).
It follows that
	\[
	\alpha((\mathrm{map}(C_{p^n}, \{1, ..., p+1\}), \gamma^*))
	= -(p-1)! \cdot \binom{p+1}{p}
	= 1 \in \mathbb{Z}/p.
	\]
This completes the proof of Lemma \ref{show-y-good} and hence
of Lemma \ref{find-x}.
\end{proof}

%% file: TheProof.tex
\section{Proof of the main theorem}\label{sec:proof}

We are now ready to prove the main theorem, which we recall here
for convenience.

\begin{theorem}\label{thm:main} Let $G = C_{p^n}$
and let $S^1 \to \mathrm{BGL}_1(S^0_{(p)})$
be adjoint to $1-p \in \pi_0^{G}S^0_{(p)}$. Denote
by $\mu: \Omega^{\lambda}S^{\lambda+1}\to \mathrm{BGL}_1(S^0_{(p)})$
the extension of this map over the $\lambda$-loop space.
Then the Thom class
	\[
	\left(\Omega^{\lambda}S^{\lambda+1}\right)^{\mu}
	\longrightarrow \mathrm{H}\underline{\mathbb{F}}_p
	\]
is an equivalence of $G$-spectra.
\end{theorem}

Before the proof, we record a well-known computation
(the $p=2$ case is proven in \cite[Proposition 3.18]{HHR}
and the odd primary proof is much the same).

\begin{lemma}\label{lem:em-geo-fixed} For $G = C_{p^n}$ and $p$
odd, we have
	\begin{align*}
	\pi_*\mathrm{H}\underline{\mathbb{Z}}^{\Phi G}
	&= \mathbb{F}_p[t], \, |t|=2,\\
	\pi_*\mathrm{H}\underline{\mathbb{F}}_p^{\Phi G}&=
	\mathbb{F}_p[t] \otimes \Lambda(s), |s|=1
	\end{align*}
When $p=2$, the second computation becomes
	\[
	\pi_*\mathrm{H}\underline{\mathbb{F}}_2^{\Phi G}
	= \mathbb{F}_2[s], \, |s|=1.
	\]
\end{lemma}

We also will need the corresponding result about our Thom spectrum. 

\begin{lemma}\label{lem:htpy-thom-geo-fixed}
Let $X$ denote the Thom spectrum $(\Omega^\lambda S^\lambda)^{\mu}$.  Then the homotopy group $\pi_k(X^{\Phi G})$ of the geometric fixed points is isomorphic to $\mathbb{F}_p$ for each $k \ge 0$.
\end{lemma}

\begin{proof}
As explained in Remark \ref{rmk:a2-str}, we have equipped $X$ with the structure of an $\mathbb{A}_2$-algebra in $\mathbb{E}_\lambda$-algebras and hence the norm map
$$\mathrm{N}^G(X) \longrightarrow X$$
is a map of $\mathbb{A}_2$-algebras.
In particular, $X^{\Phi G}$ is a module over $(\mathrm{N}^G(X))^{\Phi G}
\simeq (\mathrm{N}^G(\mathrm{H}\mathbb{F}_p))^{\Phi G} \simeq 
\mathrm{H}\mathbb{F}_p$.
Since $\mathrm{H}\mathbb{F}_p$ is a field spectrum, $X^{\Phi G}$ splits as a wedge of suspensions of $\mathrm{H}\mathbb{F}_p$.  The homotopy groups of $X^{\Phi G}$ are then determined by the homology groups of $X^{\Phi G}$.
By the Thom isomorphism, we have $\mathrm{H}_*(X^{\Phi G}) \simeq 
\mathrm{H}_*(\Omega^2 S^3 \times \Omega S^2)$, and the result follows.
\end{proof}

Finally, let $\mathcal{P}$ denote the family of proper subgroups of $G$.  For any $G$-spectrum $E$, we let let $E_{h\mathcal{P}}$
denote the spectrum $(E \wedge \mathrm{E}\mathcal{P}_+)^G$
(see, for example, \cite[2.5.2]{HHR}), and we note the natural isotropy separation sequence
\[E_{h\mathcal{P}} \to E^{G} \to E^{\Phi G}.\]
We will need the following final lemma before proving the main theorem:

\begin{lemma} \label{alpha-test}
Let $p$ be an odd prime so that
$\pi_1S^0_{h\mathcal{P}} \to \pi_1^{C_{p^n}}S^0$ is a $p$-local equivalence.
Let $\alpha:\pi_1^{G} S^0_{(p)} \to \mathbb{Z}/p$ denote the function defined at the beginning of Section \ref{sec:normII}.  Then an element $z \in \pi_1(S^0_{h\mathcal{P}})$
maps to a generator in $\pi_1(\underline{\mathbb{F}}_p)_{h\mathcal{P}}$
if and only if $\alpha(z)$ is nonzero.
\end{lemma}

\begin{proof}[Proof of Lemma \ref{alpha-test}]
If $Y$ is any $C_{p^n}$-spectrum, then we have
	\[
	Y_{h\mathcal{P}} \simeq Y^{C_{p^{n-1}}}_{h\mathrm{Aut}(C_{p^n}/C_{p^{n-1}})},
	\]
since every proper subgroup is contained in $C_{p^{n-1}}$. The map
	\[
	\bigoplus_{0\le k\le n-1} S^0_{h\mathrm{Aut}(C_{p^{n-1}}/C_{p^k})}
	\simeq
	(S^0)^{C_{p^{n-1}}} \to \underline{\mathbb{F}}_p^{C_{p^{n-1}}}=\mathbb{F}_p
	\]
sends every summand to zero except the summand corresponding
to the trivial $C_{p^{n-1}}$-orbit, $[C_{p^{n-1}}/C_{p^{n-1}}]$, which
maps to the generator. It follows that the induced map on homotopy orbits
factors through
	\[
	S^0_{h\mathrm{Aut}(C_{p^n}/C_{p^{n-1}})} \to 
	(\mathbb{F}_p)_{h\mathrm{Aut}(C_{p^n}/C_{p^{n-1}})}.
	\]
This is an isomorphism on $\pi^{C_{p^n}}_1$, which completes the proof.
\end{proof}

\begin{proof}[Proof of the main theorem] We prove the theorem by induction on $n$. When $n=0$,
this is the non-equivariant result of Hopkins-Mahowald.
For the induction hypothesis we assume that the map
	\[
	\alpha: X:=\left(\Omega^{\lambda}S^{\lambda+1}\right)^{\mu}
	\longrightarrow \mathrm{H}\underline{\mathbb{F}}_p
	\]
is an equivalence after restriction to $C_{p^{n-1}}$, and
we assume that $n\ge 1$ from now on.  We must prove that the induced map
	\[
	X^{\Phi C_{p^n}} \to \underline{\mathbb{F}}_p^{\Phi C_{p^n}}
	\]
is an equivalence.

The homotopy groups of the target are:
	\[
	\pi_*\underline{\mathbb{F}}_p^{\Phi C_{p^n}} =
	\begin{cases}
	\Lambda(s) \otimes \mathbb{F}_p[t] & p\text{ odd}\\
	\mathbb{F}_2[s] & p=2
	\end{cases}
	\]
where $|s|=1$ and $|t|=2$. Additively, these agree with the homotopy groups
of $X^{\Phi C_{p^n}}$. So we need only show that
	\[
	\pi_1X^{\Phi C_{p^n}} \to \pi_1\underline{\mathbb{F}}_p^{\Phi C_{p^n}}
	\]
is surjective and, at odd primes, that
	\[
	\pi_2X^{\Phi C_{p^n}} \to \pi_2\underline{\mathbb{F}}_p^{\Phi C_{p^n}}
	\]
is surjective. From the isotropy separation sequence and the induction
hypothesis, we have a diagram of exact sequences
for $i\ge 1$:
	\[
	\xymatrix{
	\pi_iX^{\Phi C_{p^n}} \ar[r] \ar[d]& 
	\pi_{i-1}X_{h\mathcal{P}} \ar[r]\ar[d]^{\cong} & 
	\pi_{i-1}X^{C_{p^n}}\ar[d] \ar[r]&
	\pi_{i-1}X^{\Phi C_{p^n}}\ar[d]\\
	\pi_i\underline{\mathbb{F}}_p^{\Phi C_{p^n}} \ar[r]_{\cong} &
	\pi_{i-1}(\underline{\mathbb{F}}_p)_{h\mathcal{P}} \ar[r]^0 &
	\pi_{i-1}\underline{\mathbb{F}}_p\ar[r] & 
	\pi_{i-1}\underline{\mathbb{F}}_p^{\Phi C_{p^n}}
	}
	\]
Applying this to the case $i=1$, and using that
$\pi_0X^{C_{p^n}} \to \mathbb{F}_p$ is an isomorphism (Corollary \ref{cor:pi-0}), 
we deduce that the map $\pi_{0}X_{h\mathcal{P}} \to 
\pi_{0}X^{C_{p^n}}$ is zero.  Hence the map
	\[
	\mathbb{Z}/p=\pi_1X^{\Phi C_{p^n}} \to \pi_0X_{h\mathcal{P}}=\mathbb{Z}/p
	\]
is surjective, and so it is an isomorphism. It follows
that $\pi_1X^{\Phi C_{p^n}} \to \pi_1\underline{\mathbb{F}}_p^{\Phi C_{p^n}}$
is an isomorphism, which completes the proof when $p=2$.
When $p$ is odd, we continue as follows. First, the fact that
$\pi_1X^{\Phi C_{p^n}} \to \pi_0X_{h\mathcal{P}}$ is an isomorphism means that
	\[
	\xymatrix{
	\mathbb{Z}/p=
	\pi_{1}(\underline{\mathbb{F}}_p)_{h\mathcal{P}}
	\cong \pi_1X_{h\mathcal{P}} \to \pi_1X^{C_{p^n}}.
	}
	\]
is surjective. 

By Lemma \ref{alpha-test} and Proposition \ref{find-x}, the generator of the source maps to zero in the target,
so we deduce that $\pi_1X^{C_{p^n}} = 0$, and hence that, in the diagram
	\[
	\xymatrix{
	\pi_2X^{\Phi C_{p^n}} \ar[r]\ar[d] & \pi_1 X_{h\mathcal{P}}\ar[d]^{\cong}\\
	\pi_2\underline{\mathbb{F}}_p^{\Phi C_{p^n}} \ar[r]_{\cong} &
	\pi_1(\underline{\mathbb{F}}_p)_{h\mathcal{P}}
	}
	\]
the top horizontal arrow is surjective (hence an isomorphism). Thus,
the left vertical arrow is surjective (hence an isomorphism). 
\end{proof}

%% file: Epilogue.tex
\section{Concluding remarks}\label{sec:epilogue}

\subsubsection*{The integral Eilenberg-MacLane spectrum}
Define $S^{\lambda+1}\langle \lambda+1\rangle$
as the fiber of the unit map
	\[
	S^{\lambda+1} \to K(\underline{\mathbb{Z}}, \lambda+1)
	:= \Omega^{\infty}
	\left(\Sigma^{\lambda+1}\mathrm{H}\underline{\mathbb{Z}}\right).
	\]
Then we have the following result.
	\begin{theorem}\label{thm:Z} There is an equivalence of
	$\mathbb{E}_{\lambda}$-algebras
		\[
		\left(
		\Omega^{\lambda}(S^{\lambda+1}\langle \lambda+1\rangle)\right)^{\mu}
		\simeq \mathrm{H}\underline{\mathbb{Z}}_{(p)}.
		\]
	\end{theorem}
\begin{proof} We argue as in
Antol\'{\i}n-Camarena-Barthel 
\cite[\S5.2]{omar-toby}, though we need
not develop all the technology present there.
We have a fiber sequence
	\[
	\Omega^{\lambda}S^{\lambda+1}\langle \lambda+1\rangle
	\to \Omega^{\lambda}S^{\lambda+1} \to S^1.
	\]
Decomposing $S^1$ into a 0-cell and a 1-cell, and
trivializing the fibration on each cell, produces a decomposition
of the Thom spectrum $\left(\Omega^{\lambda}S^{\lambda+1}\right)^{\mu}$
as a cofiber
	\[
	\xymatrix{
	\left(\Omega^{\lambda}S^{\lambda+1}\langle \lambda+1\rangle
	\right)^{\mu}
	\ar[r]^{x} &
	\left(\Omega^{\lambda}S^{\lambda+1}\langle \lambda+1\rangle
	\right)^{\mu}
	\ar[r] &
	\left(\Omega^{\lambda}S^{\lambda+1}\right)^{\mu}}
	\simeq \mathrm{H}\underline{\mathbb{F}}_p.
	\]
Each of these Thom spectra came from bundles classified by
$\mathbb{A}_2$-maps, which is enough to ensure
that the map $x$ induces a map 
$\pi_*\left(\Omega^{\lambda}S^{\lambda+1}\langle \lambda+1\rangle\right)^{\mu}
\to
\pi_*\left(\Omega^{\lambda}S^{\lambda+1}\langle \lambda+1\rangle
\right)^{\mu}$ of modules over
$\pi_0\left(\Omega^{\lambda}S^{\lambda+1}\langle \lambda+1\rangle
\right)^{\mu}$. In particular, on homotopy the map corresponds
to multiplication by some element $x \in \pi_0\left(\Omega^{\lambda}S^{\lambda+1}\langle \lambda+1\rangle
	\right)^{\mu}$.
Arguments similar to those in the proof
of the main theorem show that 
	\[\underline{\pi}_0
\left(\Omega^{\lambda}S^{\lambda+1}\langle \lambda+1\rangle
\right)^{\mu}\simeq\underline{\mathbb{Z}}_{(p)},\] so we must have
$x = p$. The result follows
from Nakayama's lemma once 
one argues that the genuine fixed point spectra
have finitely generated homotopy groups
in each degree. (For example, isotropy separation
reduces us to the corresponding statement on
geometric fixed points, where it follows from the Thom isomorphism.)
\end{proof}

\begin{remark} The map
$S^{\lambda+1} \to K(\underline{\mathbb{Z}}, \lambda+1)$ deloops to a map
$\mathbb{H}P^{\infty} \to K(\underline{\mathbb{Z}},
\lambda+2)$. It follows from Theorem 
\ref{thm:main-even} that the equivalence above
is one of $\mathbb{E}_{\lambda+1}$-algebras
when $p=2$.
\end{remark}

\begin{remark} Unlike the classical case,
it is unclear whether the statement globalizes to
a construction of $\mathrm{H}\underline{\mathbb{Z}}$. Our methods do
not construct
$\mathrm{H}\underline{\mathbb{F}}_{\ell}$ as a Thom spectrum
when $\ell$ does not divide the order of $G$.
\end{remark}

\subsubsection*{Questions}
We conclude with a few open-ended questions.

\begin{question}
Is $\mathrm{H}\underline{\mathbb{F}}_p$ a Thom spectrum for any group $G$ that is not cyclic of $p$-power order?  It seems plausible that this is so for dihedral groups, as was suggested to the authors by Stefan Schwede.  Can obstructions be found for other $G$?
\end{question}

\begin{question} Calculations indicate that the spectrum
$\left(\Omega^{\lambda}S^{\lambda+1}\right)^{\mu}$
is not $\mathrm{H}\underline{\mathbb{F}}_p$ for
the groups $C_n$ when $n$ is not a power of $p$.
What can be said about $\left(\Omega^{\lambda}S^{\lambda+1}\right)^{\mu}$ as an $S^1$ or $O(2)$-equivariant spectrum?  At least one expects an interesting $C_{p^{\infty}}$-spectrum.
\end{question}

\begin{question} One of Mahowald's motivations for proving the equivalence
$(\Omega^2S^3)^{\mu} \simeq \mathrm{H}\mathbb{F}_2$ is that
the left hand side carries a natural filtration due to Milgram and May.
This produces a filtration of $\mathrm{H}\mathbb{F}_2$ by
spectra which turn out to be the Brown-Gitler spectra
of \cite{brown-gitler} (see \cite{brown-peterson, cohen, hunter-kuhn}). 
The $G$-space $\Omega^{\lambda}S^{\lambda+1}$ also carries
the arity filtration from the $\mathbb{E}_{\lambda}$-operad,
so we could \emph{define} equivariant Brown-Gitler spectra
using this filtration.  It would be interesting to know if these spectra are of any use.

In the case $G = C_2$ there are
two different operadic filtrations of $\Omega^{2\sigma}S^{2\sigma+1}
\simeq \Omega^{\rho}S^{\rho+1}$. This leads to two different notions
of Brown-Gitler spectra. How are they related?
\end{question}

\begin{question} What are the Thom spectra obtained
by killing other natural elements in the Burnside ring
in a highly structured manner?
What is the free $\mathbb{E}_{\lambda}$-algebra
in $C_p$-spectra with
$[C_p]=0$?
\end{question}

\begin{question}
Can the $C_{p^n}$-equivariant dual Steenrod algebra $\mathrm{H}\underline{\mathbb{F}}_p \wedge \mathrm{H}\underline{\mathbb{F}}_p$ be profitably studied via its equivalence with $\mathrm{H}\underline{\mathbb{F}}_p \wedge \Sigma^{\infty}_+ \Omega^{\lambda} S^{\lambda+1}$?
\end{question}

%% file: appendix.tex
\section{Proof of Theorem \ref{thm:non-eqvt}} \label{sec:appendix}

For convenience, we recall the statement of Theorem \ref{thm:non-eqvt}, which this appendix is devoted to proving.  The result is entirely non-equivariant.

\begin{thm*}
Let $S^0_p$ denote the $p$-complete sphere spectrum, and suppose that $p>2$.  Then there is no triple loop map
$$X \longrightarrow \mathrm{BGL}_1(S^0_{p}),$$
for any triple loop space $X$, that makes $\mathrm{H}\mathbb{F}_p$ as a Thom spectrum.
\end{thm*}

\begin{remark}
It follows also that $\mathrm{H}\mathbb{F}_p$ is not a triple loop Thom spectrum over the $p$-local sphere spectrum.  If it were, then the composition
$$X \longrightarrow \mathrm{BGL}_1(S^0_{(p)}) \longrightarrow \mathrm{BGL}_1(S^0_p)$$
would provide a counterexample to the above.
\end{remark}

\begin{proof}
Suppose, for the sake of contradiction, that such a triple loop map 
$$X \longrightarrow \mathrm{BGL}_1(S^0_p)$$
exists.  The Thom isomorphism then implies that 
$$\mathrm{H}\mathbb{F}_p \smsh \Sigma^{\infty}_+ X \simeq \mathrm{H}\mathbb{F}_p \smsh \mathrm{H}\mathbb{F}_p$$
as $\mathrm{H}\mathbb{F}_p$--$\mathbb{E}_3$-algebras.

In particular, by Theorem \ref{thm:Hop}, 
$$\mathrm{H}\mathbb{F}_p \smsh \Sigma^{\infty}_+ X \simeq \mathrm{H}\mathbb{F}_p \smsh \Sigma^{\infty}_+ \Omega^2 S^3$$
as $\mathrm{H}\mathbb{F}_p$--$\mathbb{E}_2$-algebras, and the latter object is the free $\mathrm{H}\mathbb{F}_p$--$\mathbb{E}_2$-algebra on a class in degree $1$.

The Hurewicz theorem gives a map $S^1 \rightarrow X$, which extends to a double-loop map $\Omega^2 S^3 \rightarrow X$, and the above discussion implies that this double loop map is a homology isomorphism.  Thus, the $p$-completion of $X$ is the $p$-completion of $\Omega^2 S^3$, as a double loop space.

Transporting the $\mathbb{E}_3$-algebra structure on $X$ yields an $\mathbb{E}_3$-algebra structure on the $p$-completion of $\Omega^2 S^3$, extending the usual $\mathbb{E}_2$-algebra structure.  The theorems of Dwyer, Miller, and Wilkerson \cite{dwyer-miller-wilkerson} show that there is a unique such $\mathbb{E}_3$-algebra structure, and so the $p$-completion of $\mathrm{B}^3 X$ must be the $p$-completion of $\mathbb{H}P^{\infty}$.

Now, the composite 
$$X \longrightarrow \mathrm{BGL}_1(S^0_p) \longrightarrow \mathrm{BGL}_1(\mathrm{H}\mathbb{F}_p)$$
is null, and it follows that there is a factorization through the fiber $F$ of $\mathrm{BGL}_1(S^0_p) \longrightarrow \mathrm{BGL}_1(\mathrm{H}\mathbb{F}_p)$.
The equivalence $\mathbb{Z}_p^{\times} \cong \mu_{p-1} \times \mathbb{Z}_p$ implies that the homotopy groups of $F$ are $p$-complete.  Thus, with $\mathbb{H}P^{\infty}_p$ denoting the $p$-completion of $\mathbb{H}P^{\infty}$, there is a commuting diagram
$$
\begin{tikzcd}
 & \mathrm{B}^3X \arrow{r} \arrow{d} & B^4GL_1(S^0_p). \\
 \mathbb{H}P^{\infty} \arrow{r} & \mathbb{H}P^{\infty}_p \arrow{ru}
\end{tikzcd}
$$
In particular, there is a triple-loop map
$$\Omega^2 S^3 \longrightarrow \Omega^3 \mathbb{H}P^{\infty} \longrightarrow \mathrm{BGL}_1(S^0_p)$$
with Thom spectrum equivalent (at least after $p$-completion) to $\mathrm{H}\mathbb{F}_p$.

The underlying double loop map is determined by a class in $1+p\alpha \in \pi_3(\mathrm{B}^3\mathrm{GL}_1(S^0_p)) \cong \mathbb{Z}^{\times}_p$.  Our original assumption, made for the sake of contradiction, is reduced to the assertion that a dashed arrow exists the diagram below:
$$
\begin{tikzcd}[column sep = 5.0em]
S^4 \arrow{d} \arrow{r}{1+p\alpha} & \mathrm{B}^4\mathrm{GL}_1(S^0_p). \\
\mathbb{H}P^{\infty} \arrow[dashed]{ur}
\end{tikzcd}
$$

We will show this to be impossible by proving the non-existence of a solution to the weaker lifting problem
$$
\begin{tikzcd}
S^4 \arrow{d} \arrow{r}{1+p\alpha} & \Sigma^{\infty} 
\mathrm{B}^4\mathrm{GL}_1(S^0_p) \arrow{r}{\ell} & \Sigma^4 L_{K(1)} S^0,  \\
\Sigma^{\infty} \mathbb{H}P^{\infty} \arrow[dashed]{urr} 
\end{tikzcd}
$$
where $L_{K(1)} S^0$ is $K(1)$-local sphere spectrum and $\ell$ is the Rezk logarithm \cite{rezk}.  We first calculate the composite 
$$S^4 \stackrel{1+p\alpha}{\longrightarrow} \mathrm{B}^4\mathrm{GL}_1(S^0_p) \stackrel{\ell}{\longrightarrow} \Sigma^4 L_{K(1)} S^0,$$
using Rezk's formula \cite[Theorem 1.9]{rezk} for the logarithm at odd primes:
$$\ell(1+p\alpha)=log(1+p\alpha) - \frac{1}{p} log(1+p\alpha).$$

If $\alpha$ were not a $p$-adic unit, then the composite $\Omega^2 S^3 \longrightarrow \mathrm{BGL}_1(S^0) \longrightarrow \mathrm{BGL}_1(\mathrm{H}\mathbb{Z}/p^2)$ would be null as a $2$-fold loop map, providing a ring map $\mathrm{H}\mathbb{F}_p \longrightarrow \mathrm{H}\mathbb{Z}/p^2$. Since this is absurd, $\alpha$ must be a $p$-adic unit, and we learn that $\ell(1+p\alpha)$ is also a $p$-adic unit. 

Without loss of generality, then, we are reduced to showing the impossibility of the following lifting problem:
$$
\begin{tikzcd}
S^4 \arrow{r}{1} \arrow{d} & \Sigma^4 L_{K(1)} S^0, \\
\Sigma^{\infty} \mathbb{H}P^{\infty} \arrow[dashed]{ur}
\end{tikzcd}
$$
where $1$ is the unit of the ring spectrum $L_{K(1)} S^0$.

Let $\mathrm{KU}_p$ denote $p$-complete complex $K$-theory.  Recall that the composite
$$L_{K(1)} S^0 \longrightarrow \mathrm{KU}_p \stackrel{\psi^q-1}{\longrightarrow} \mathrm{KU}_p$$
is null for any Adams operation $\psi^q$ with $q$ relatively prime to $p$.
Since $p$ is odd, to finish the problem it will suffice for us to show that no element of $\mathrm{KU}_p^4(\mathbb{H}P^{\infty})$ simultaneously:
\begin{enumerate}
\item Restricts to the unit in $\mathrm{KU}_p^4(S^4)$.
\item Is invariant under the action of $\psi^2$.
\end{enumerate}

Now, $$\mathrm{KU}_p^*(\mathbb{H}P^{\infty}) \cong \mathbb{Z}_p\llbracket e\rrbracket[\beta^{\pm}],$$ where $|e|=0$ and $\beta$ is the Bott class in degree $-2$.  Of course, $\psi^2(\beta)=2\beta$, and it will be necessary also to understand $\psi^2(e)$.

Remembering that $\mathbb{H}P^{\infty}$ is $\mathrm{BSU}(2)$, we may calculate $\psi^2(e)$ by determining the restriction of $e$ along the inclusion of the maximal torus $\mathrm{B}S^1 \longrightarrow \mathrm{BSU}(2)$.  Indeed, $\mathrm{KU}_p^*(\mathrm{B}S^1) \cong 
\mathbb{Z}_p\llbracket x\rrbracket [\beta^{\pm 1}]$, where $x=L-1$.  On the other hand, $e=V-2$, where $V$ is the standard representation of 
$\mathrm{SU}(2)$ on $\mathbb{C}^2$.  The restriction of $e$ is thus $L+L^{-1}-2$, where 
$$L^{-1}=(x+1)^{-1} = 1 - x + x^2 - x^3 + \cdots.$$
Since 
$$\psi^2(L+L^{-1}-2)=L^2+ L^{-2} -2= (x+1)^2 +\frac{1}{(x+1)^2}-2 = \left(x+1+\frac{1}{x+1} -2\right)^2+4 \left(x+1+\frac{1}{x+1} -2\right),$$
we calculate that 
$$\psi^2(e)=e^2+4e.$$

An element of $\mathrm{KU}^4(\mathbb{H}P^{\infty})$ is of the form $\beta^{-2} P(e)$, where $P(e)$ is a power series in $\mathbb{Z}_p\llbracket 
e\rrbracket$.  The lifting problem in question is equivalent to finding a power series $P(e)=e + c_2 e^2 + \cdots$ such that
$$P(e) = 2^{-2} P\left( \psi^2(e) \right).$$
Using the calculations above, this can be rewritten as the relation
$$4 P(e) = P(e^2+4e).$$
The relation
$$4(e+c_2e^2+c_3e^3+\cdots) = (e^2+4e)+c_2(e^2+4e)^2+c_3(e^2+4e)^3 + \cdots$$
inductively determines each $c_i$, given $c_2,\cdots,c_{i-1}$, according to the formula
$$c_i = \frac{2}{(2i)!} \prod_{j=2}^{i} \left(-(j-1)^2\right).$$
In particular, this formula does not yield a $p$-adic integer for $i=\frac{p+1}{2}$, implying that there is no lift through $\mathbb{H}P^{\frac{p+1}{2}}$.
\end{proof}

\begin{remark}
The Adams conjecture provides a map from the connective cover of the $K(1)$-local sphere spectrum into $\mathrm{gl}_1(S^0_{(p)})$.
 Using a variant of this due to Bhattacharya and Kitchloo \cite{bhattacharya-kitchloo}, it is possible to construct maps $\mathbb{H}P^{k} \longrightarrow \mathrm{B}^4 \mathrm{GL}_1(S^0_{(p)})$, for small values of $k$.  Indeed, \cite{bhattacharya-kitchloo} employs arguments very similar to the ones above in order to produce multiplicative structures on Moore spectra.  The authors believe, but have not verified, that it is possible to equip the map $S^3 \stackrel{1-p}{\longrightarrow} \mathrm{B}^3 
\mathrm{GL}_1(S^0_{(p)})$ with an $\mathbb{A}_{\frac{p-1}{2}}$-algebra structure in this manner.
\end{remark}

\begin{remark}
It is well-known that the integral Eilenberg--Maclane spectrum $\mathrm{H}\mathbb{Z}_{(p)}$ is the Thom spectrum of a double loop composite
$$\Omega^2 (S^3 \langle 3 \rangle) \longrightarrow \Omega^2 S^3 \longrightarrow \mathrm{BGL}_1(S^0_p).$$
One could attempt to refine this to a triple loop map, using the equivalence $\Omega \mathbb{H}P^{\infty} \langle 4 \rangle \simeq S^3 \langle 3 \rangle$.
The same obstruction as above proves that this strategy cannot work at odd primes, because the map $$\mathbb{H}P^{\infty} \langle 4 \rangle \longrightarrow \mathbb{H}P^{\infty}$$ is a $K(1)$-local equivalence.
\end{remark}